%% file: icml_main.tex
\documentclass[nohyperref]{article}

\usepackage{microtype}
\usepackage{graphicx}
\usepackage{subfigure}
\usepackage{booktabs} 
\usepackage[OT1]{fontenc} 
\usepackage{hyperref}
\usepackage{algorithm2e}
\RestyleAlgo{ruled}

\usepackage{multirow}
\usepackage{tabularx}
\usepackage[mathscr]{eucal}
\usepackage[accepted]{icml2024}

\usepackage{amsmath}
\usepackage{amssymb}
\usepackage{mathtools}
\usepackage{amsthm}

\usepackage[capitalize,noabbrev]{cleveref}
\input{notation.tex}

\icmltitlerunning{Bootstrapping Fisher Market Equilibrium and First-Price Pacing Equilibrium}

\begin{document}

\twocolumn[
\icmltitle{Bootstrapping Fisher Market Equilibrium and First-Price Pacing Equilibrium}



\icmlsetsymbol{equal}{*}

\begin{icmlauthorlist}
\icmlauthor{Luofeng Liao}{sch}
\icmlauthor{Christian Kroer}{sch}
\end{icmlauthorlist}

\icmlaffiliation{sch}{IEOR, Columbia University}

\icmlcorrespondingauthor{Luofeng Liao}{ll3530@columbia.edu}

\icmlkeywords{First-price auction, pacing equilibria}

\vskip 0.3in
]


\printAffiliationsAndNotice{}  

\input{maintext.tex}
\bibliography{refs.bib}
\bibliographystyle{icml2024}
\newpage
\appendix
\onecolumn
\input{tech_lemma.tex}
\input{proofs.tex}
\end{document}

%% file: notation.tex
\usepackage{mathtools}       
\usepackage{amsmath}
\usepackage{amsthm}
\usepackage{accents}
\usepackage{thmtools, thm-restate}
\usepackage{amsmath,amssymb,amsthm}
\usepackage{algorithm2e} 
\usepackage{makecell}
\usepackage{lscape}
\usepackage[shortlabels]{enumitem}
\setlist[itemize]{topsep=0pt,}
\usepackage{pifont}
\newcommand{\cmark}{ {\color{teal} \ding{51}\;}}%
\newcommand{\xmark}{ {\color{red}{\ding{55}}\;}}%
\hypersetup{
colorlinks   = true, 
urlcolor     = blue, 
linkcolor    = blue, 
citecolor   = blue 
}

\usepackage{dsfont} 
\usepackage{xcolor} 
\usepackage{cleveref}
\usepackage{esvect}

\usepackage{tabularx}
\usepackage{booktabs}
\usepackage[labelfont=bf,format=plain,justification=raggedright,singlelinecheck=false]{caption}

\usepackage{calrsfs}

\newcommand{\pmr}{pacing multiplier\,}
\newcommand{\pmrs}{pacing multipliers\,}
\def\given{\,|\,}

\def\tr{\mathop{\text{tr}}\kern.2ex}

\def \cond {{\,|\,}}
\def \gam {{ \gamma }}
\def \indi {{ \mathds{1} }}

\def\R{{\mathbb R}}

\def\Rn{{\mathbb R^n}}
\def\Rnp{{\mathbb R^n_+}}

\def\Rnpp{{\mathbb R^n_{++}}}

\def\P{{\mathbb P}}
\def\E{{\mathbb E}}
\def\B{{\mathbb B}}

\def\cC{{\mathcal{C}}}

\def\cH{{\mathscr{H}}}
\def\cJ{{\mathscr{J}}}
\def\cS{{\mathcal{S}}}

\def\cL{{\mathcal{L}}}
\def\cN{{\mathcal{N}}}

\newtheorem{Assumption}{Assumption}
\newtheorem{example}{Example}
\newtheorem{remark}{Remark}
\newtheorem{claim}{Claim}

\newtheorem{lemma}{Lemma}
\newtheorem{defn}{Definition}

\newlist{enumconditions}{enumerate}{1} 
\setlist[enumconditions]{label = \thelemma.\alph*}
\crefname{enumconditionsi}{Cond.}{Conditions}

\newlist{enumlmresult}{enumerate}{1} 
\setlist[enumlmresult]{label = \thelemma.\arabic*}
\crefname{enumlmresulti}{Part}{Parts}

\newlist{enumdef}{enumerate}{1} 
\setlist[enumdef]{label = \thedefn.\arabic*}
\crefname{enumdefi}{Cond.}{Cond.}

\crefname{Assumption}{Assumption}{Assumptions}
\renewcommand*{\theAssumption}{\arabic{Assumption}}
\crefname{defn}{Def}{Defs}
\crefname{example}{Ex}{Examples}

\newlist{enumthmresult}{enumerate}{1} 
\setlist[enumthmresult]{label = \thetheorem.\arabic*}
\crefname{enumthmresulti}{Thm}{Theorems}

\crefname{equation}{Eq}{Eqs}

\newlist{enumassumption}{enumerate}{1} 
\setlist[enumassumption]{label = \theAssumption.\arabic*}
\crefname{enumassumptioni}{Condition}{Conditions}

\crefname{appendix}{App}{Appendices}
\crefname{theorem}{Thm}{Theorems}
\crefname{section}{Sec}{Sections}
\def \hat {\widehat}
\newtheorem{theorem}{Theorem}
\renewcommand*{\thetheorem}{\arabic{theorem}}

\def \cd {{ \cdot }}
\def \deltasti {{\delta^*_i}}
\def \deltast {\delta^*}

\def \dwt { \delta_{t}}
\def \dtt { \epsilon_{t}}

\def \tdiff {{\hat \nabla\sq_{k\ell,\epsilon}}}
\def \tdifft {{\hat \nabla\sq_{k\ell,\eta_t}}}

\def \Ic {{I_<}}

\def \epi {{\mathrm{epi}\,}}

\def \REVst {{\small \textsc{{REV}}  }^* }

\def \FPPE {{  \mathsf{{FPPE}}}}
\def \oFPPE {{ \widehat{ \mathsf{{FPPE}}}}}
\def \LFM {{  \mathsf{{LFM}}}}
\def \oLFM {{  \widehat{ \mathsf{{LFM}} } }}

\def \CTR {{ \small \textsc{{CTR}}}}
\def \click {{ \small \textsc{{click}}}}
\def \AE {{ \small \textsc{{Adexchange}}}}
\def \RG{{ \small \textsc{{Region}}}}
\def \TG{{ \small \textsc{{Usertag}}}}

\def \pay {{ \small  \textsf{{pay}}}}

\def \boot {b}

\def \b { \beta }

\def \ust {{ u^* }}

\def \ugami { u^\gamma_i }

\def \usti {{ u^*_i }}

\def \nmt {\|_2}

\def \vi {{v_i}}
\def \vithetau {{ v_i(\theta^\tau) }}
\def \vithe {{ v_i(\theta) }}
\def \vitau {{ v_i^\tau }}
\def \vbar {{ \bar v }}

\def \vbart {{ \bar v^t }}
\def \vbartboot {{ \bar v^{t, \boot} }}

\def \var {{ \operatorname {Var} }}
\def \cov {{ \operatorname {Cov} }}

\def \sumtau {{\sum_{\tau=1}^{t}}}

\def \sumiton {{ \sum_{i=1}^n }}
\def \sumi {{ \sum_{i} }}

\def \ptau {{ p ^ \tau }}

\def \xtaui {{ x^\tau_i }}

\def \t {\theta}
\def \thetau {{ \theta^\tau }}

\def \eps{{ \epsilon }}
\def \ept { \epsilon_t }

\def \pst {p^*}

\def \betab { \beta ^ \boot}

\def \betast {\beta^*}

\def \betasti {{\beta^*_i}}

\def \betai {{\beta_i}}
\def \betak {{\beta_k}}
\def \betagam {\beta^\gamma}

\def \betagami {\beta^\gamma_i}

\def \xitau {{x_i^\tau}}
\def \xst {x^*}
\def \xsti {{x^*_i}}

\def \inv {^{{-1}}}

\def \Ihatp {{ \hat I _ = }}
\def \pinv {^{\dagger}}
\def \sq {^{2}}
\def \tp {^{\scriptscriptstyle\mathsf{T}}}
\def \st {^{*}}
\def \ot {^{\otimes 2}}

\def \Diag {{ \operatorname*{Diag}}}

\newcommand{\defeq}{ = }

\newcommand*\diff{\mathop{}\!\mathrm{d}}

\def \toprob {{ \,\overset{{p}}{\to}\, }}
\def \todepi {{ \,\stackrel{\mathsf{epi}}{\rightsquigarrow}\,}}
\def \toepi {{ \,\overset{\mathsf{epi}}{\longrightarrow}\, }}
\def \tod {{ \,\overset{{d}}{\to} \,}}
\def \topw {\,{\stackrel{p}{\rightsquigarrow}}\,}
\def \invtroot {t^{-1/2}}

\DeclareMathOperator*{\argmax}{arg\,max}
\DeclareMathOperator*{\argmin}{arg\,min}

\def \tcH {\tilde{\cH}}
\def \tG {\tilde{G}}

\def \Pbt {{P^\boot_t}}
\def \Pexbt {{P^{\mathsf{ex}, \boot}_t}}
\def \Pt {P_t}

%% file: maintext.tex
\begin{abstract}
    The linear Fisher market (LFM) is a basic equilibrium model from economics, which also has application in fair and efficient resource allocation. First-price pacing equilibrium (FPPE) is a model capturing budget-management mechanisms in  first-price auctions. 
    In certain practical settings such as advertising auctions, there is an interest in performing statistical inference over these models. A popular methodology for general statistical inference is the bootstrap procedure.
    Yet, for LFM and FPPE there is no existing theory for the valid application of bootstrap procedures.
    In this paper, we introduce and devise several statistically valid bootstrap inference procedures for LFM and FPPE.
    The most challenging part is to bootstrap general FPPE, which reduces to bootstrapping constrained M-estimators, a largely unexplored problem. 
    We are able to devise a bootstrap procedure for FPPE under mild degeneracy conditions by using the powerful tool of epi-convergence theory.
    Experiments with synthetic and semi-real data verify our theory.
\end{abstract}
\section{Introduction}

The bootstrap \citep{efron1994introduction,horowitz2001bootstrap} is 
an automatic method for producing confidence intervals in statistical estimation. 
The theory of bootstrap has been extended to many areas of statistics, such as 
models with
cube-root asymptotics \citep{cattaneo2020bootstrap,patra2018consistent},
semi-parametric models 
\citep{cheng2010bootstrap,ma2005robust}
and so on.
However, as far as we are concerned, there is no theory of bootstrap for competitive equilibrium settings. 

In this paper, we study bootstrap inference in linear Fisher market (LFM) and first-price pacing equilibrium (FPPE).
Fisher market equilibrium model has been used in the tech industry, such as
the allocation of impressions to content in certain recommender systems~\citep{robust_blog},
robust 
and fair work allocation in content review~\citep{allouah2022robust};
we refer readers to~\citet{kroer2022market} for a comprehensive review.
Outside the tech industry, Fisher market equilibria also have applications to
scheduling problems~\citep{im2017competitive},
fair course seat allocation~\citep{othman2010finding,budish2016course},
allocating donations to food banks~\citep{aleksandrov2015online},
sharing scarce compute resources~\citep{ghodsi2011dominant,parkes2015beyond,kash2014no,devanur2018new}, 
and allocating blood donations to blood banks~\citep{mcelfresh2020matching}.

FPPE is a model for budget management in online advertising platforms.
In these platforms, advertisers report advertising parameters, such as target audience, conversion locations, and budgets, and then the platform creates a proxy bidder to bid in individual auctions to maximize advertiser utilities and manage budgets. A common way to manage budgets is pacing, 
in which the platform modifies the advertiser's bids by applying a shading factor, referred to as \emph{multiplicative pacing}.
In the case where each auction is a first-price auction, FPPE captures the outcomes of pacing-based budget-management systems.
\citet{conitzer2022pacing} introduced the FPPE notion and showed that FPPE always exists and is unique. Moreover, FPPE enjoys lots of nice properties such as being revenue-maximizing among all budget-feasible pacing strategies, 
shill-proof (the platform does not benefit from adding fake bids under first-price auction mechanism) and
revenue-monotone (revenue weakly increases when adding bidders, items or budget).

Given the wide range of applications of LFM and FPPE, an inferential theory for LFM and FPPE is useful.
Bootstrap, thanks to its convenience and  conceptual simplicity, is a natural candidate as an inferential tool.
However, due to the presence of an equilibrium structure in the dataset, 
the validity of bootstrap requires careful theoretical treatments, 
and practitioners should be cautious about the use of bootstrap when data arise from market equilibrium. 
For example, in \cref{sec:fail_multi_bs} we show that in the setting of first-price auction platforms, 
the traditional multinomial bootstrap may fail to consistently estimate the distribution of interest. 
Given the simplicity of resampling, it is fair to say bootstrap has been used in auction platforms as an inferential tool.
It is thus urgent to develop a statistically valid bootstrap theory that accounts for the equilibrium effect in the data.

Our contributions are threefold.

\textbf{We characterize the full landscape of the asymptotics of FPPE.}
The limit distribution of FPPE was studied in \citet{liao2023stat} under a strict complementarity condition. Combining their results with a result of \citet{shapiro1989asymptotic}, we complete the characterization of the asymptotics of FPPE without strict complementarity, and show that it is captured by a quadratic program.
We derive a new closed-form expression for this quadratic program, and use it to derive structural insights on the limit distribution in some special cases.
Characterizing the general case of FPPE asymptotics is necessary in order to derive our bootstrap results, because we need to show that our bootstrapped distribution converges to the asymptotic distribution of FPPE.

\textbf{We develop bootstrap theory for LFM and FPPE.}
A crucial fact for LFM and FPPE is that they both have an Eisenberg-Gale (EG) convex program characterization, and our bootstrap procedures rely on this program or its quadratic approximation.
For LFM we study three types of bootstrap procedures: exchangeable bootstrap \citep{wellner1996bootstrapping}, numerical bootstrap \citep{li2020numerical} and proximal bootstrap \citep{li2020numerical}. 
For FPPE the theory is a bit involved.
We identify a bootstrap failure when some type of degenerate buyers are present in the market. 
Then different bootstrap procedures are proposed under certain assumptions on the market structures: full expenditure of budgets ($I_{=>}= \emptyset$), absence of degenerate buyers ($I_{==} =\emptyset$), or fully general FPPE.
We summarize the results in \cref{tbl1,tbl2}.

\textbf{Numerical experiments demonstrate the validity of the theory.} We provide simulations and a semi-synthetic experiment based on a real-time bidding dataset from iPinYou \citep{liao2014ipinyou}.

\begin{table}[ht]
    \centering
    \begin{tabular}{p{2.5cm}|p{2.1cm}|p{2.1cm}}
         Exchangable BS & Numerical BS   & Proximal BS                   \\ \hline
      \cmark \cref{thm:exboot_fm} & \cmark \cref{thm:nuboot_fm} & \cmark \cref{thm:prboot_fm}                            \\
    \end{tabular}
    \caption{Results for linear Fisher market.}
    \label{tbl1}
\end{table}
\begin{table}[ht]
    \begin{tabular}{p{1.4cm}|p{1.4cm}|p{1.4cm}|p{1.9cm}}
        & Num. BS 
        & Prox. BS 
        & new methods 
    \\ \cline{1-4}
        $I_{=>} = \nobreak \emptyset$  (\cref{sec:poorbuyerfppe}) 
        & \cmark \cref{thm:nuboot_fppepoor} 
        & \cmark \cref{thm:prboot_fppepoor}
        & \multirow{2}{*}{ } 
    \\ \cline{1-4}
        $I_{==} = \emptyset$ (\cref{sec:bs_with_scs})  
        & \xmark NA 
        & \xmark NA 
        & \cmark \cref{thm:adaptive_boot}  
    \\ \cline{1-4}
        general (\cref{sec:generalfppe})  
        & \xmark NA 
        & \xmark NA 
        & \cmark \cref{thm:generalFPPE_coverage}  
    \end{tabular}
    \caption{Results for first-price pacing equilibrium.
    NA means  not applicable.
    $I_{=>}= \emptyset$ means full expenditure of budgets. 
    $I_{==} =\emptyset$ means absence of degenerate buyers.}
    \label{tbl2}
\end{table}

\textbf{Notations.} 
The notation $\cN(a,\Sigma)$ stands for a multivariate Gaussian distribution with mean $a$ and covariance $\Sigma$.

We use $W = (W_1,\dots, W_t)$ to denote bootstrap weights in the paper. 
Different distributions imposed on $W$ correspond to different bootstrap resampling schemes.
In the standard multinomial bootstrap $W = (W_1,\dots, W_t)$ follows a multinomial with probabilities $(\frac1t, \dots, \frac1t)$.
In exchangeable bootstrap $W$ is exchangeable:
if for any permutation $\pi=\left(\pi_1, \ldots, \pi_t\right)$ of $(1,2, \ldots, t)$, the joint distribution of $\pi\left(W\right)=$ $\left(W_{ \pi_1}, \ldots, W_{ \pi_t}\right)$ is the same as that of $W$.
Given items $(\thetau)_\tau$, we let $P_t$ be the expectation operator $P_t f =\frac1t \sumtau  f(\thetau) $.
Given multinomial bootstrap weights $W$ and $(\thetau)_\tau$, define the operator
\begin{align}
    P^\boot_t f = \frac1t \sumtau W_\tau f(\thetau).
\end{align}
We write $\Pexbt f =  \frac1t \sumtau W_\tau f(\thetau)$ for exchangeable bootstrap weights.

\paragraph{Bootstrap Consistency}

Most of our results will be concerned with the consistency of bootstrap procedures.
To that end, we introduce the following definition of consistency.
Given $t$ data points, a bootstrap estimate $X_t$ is a function of the data $(\thetau)_{\tau=1}^t $ and bootstrap weights $W$, where the data and weights are assumed to be independent of each other.
We say the conditional distribution of $(X_t)_t$ consistently estimates the distribution $L$, denoted $X_t \topw L$, if 
\begin{align*}
    \sup_{f\in \mathrm{BL}_1}
    \big|\E[f(X_t)\cond \{\thetau\}_{1}^t] - \E_{X\sim L}[f(X)] \big | \toprob 0 \; , 
\end{align*}
where $\mathrm{BL}_1$ is the space of functions $f: \Rn \to \R$ with $\sup_{x} |f(x)| \leq 1$ and $|f(x) - f(y) | \leq \|x-y\nmt$.

We survey related work in \cref{sec:relatedwork}.

\section{Review of Fisher Market and FPPE}
Both LFM and FPPE have a set of buyers and a set of items being priced. Here we introduce some components that both models share.
We have $n$ buyers and a possibly continuous set of items $\Theta$ with {an} integrating measure $\diff \theta$.
For example, $\Theta = [0,1]$ with $\diff \theta$ being the Lebesgue measure on $[0,1]$.
Both LFM and FPPE require the following elements.
\begin{itemize}
    \item The \emph{budget} $b_i$ of buyer $i$. Let $b = (b_1,\dots, b_n)$. 
   
    \item The \emph{valuation} for buyer $i$ is a function $v_i \in L^1_+$, i.e., buyer $i$ has valuation $v_i(\theta)$ (value per unit supply) of item $\theta\in \Theta$. Let $v: \Theta \to \Rn$, $v(\theta) = [v_1(\theta),\dots, v_n(\theta)]$. We assume $\vbar = \max _i \sup_\theta  v_i(\theta)< \infty $.

    \item 
    The \emph{supplies} of items are given by a function  $ s \in L^\infty_+$, i.e., item $\theta\in \Theta$ has $s(\theta)$ units of supply. 
    Without loss of generality, we assume a unit total supply $\int_\Theta s \diff \theta = 1$. 
    Given $g:\Theta \to \R$, we let $\E[g] = \int g(\theta)s(\theta)\diff \theta$ and $\var[g] = \E[g\sq] - (\E[g])\sq$.
    Given  $t$ i.i.d.\ draws $\{ \theta^1,\dots, \theta^t\}$ 
    from $s$, let $P_t g(\cd) = \frac1t \sumtau g(\thetau)$.
\end{itemize}

Equilibria in both LFM and FPPE are characterized by an EG convex program. In both cases, the dual EG objective separates into per-item convex terms
\begin{align}
    \label{eq:def:F}
    F (\t, \b) = \max_{i\in[n]} \beta_i v_i(\theta) -  \sumiton b_i \log \beta_i \;.
\end{align}
and the population and sample EG objectives are
\begin{align}
    \label{eq:def_pop_eg}
    H(\beta) = \E[F(\t,\b)] \; , \;\; H_t(\beta) = P_t F(\cd, \beta) \; .
\end{align}
We comment on the differential structure of $f(\t,\b)=\max_i \betai\vithe$ since it plays a role in later sections. 
Function $f(\b, \t)$ is a convex function of $\b$ and its subdifferential $\partial_\b f(\b,\t)$ is the convex hull of $\{ v_i e_i: \text{index $i$ such that $\betai\vithe = \max_k \betak v_k(\t)$}  \}$, with $e_i$ being the base vector in $\Rn$.
When $\max_i \betai v_i(\t)$ is attained by a unique $i^*$, the function $f$ is differentiable. In that case, the $i$-th entry of $\nabla_\b f(\t,\b)$ is $v_i(\t)$ for $i=i^*$ and zero otherwise. 

\subsection{Linear Fisher Markets (LFM)}

In the LFM model, the goal is to divide items $\Theta$ in a fair and efficient way. It is well known that 
the competitive equilibrium from equal incomes (CEEI)
mechanism produces an allocation that is Pareto efficient, envy-free and proportional \citep{nisan2007algorithmic}.
LFM is also a useful tool for modeling competition in an economy.

We now describe the competitive equilibrium concept. 
Imagine there is a central policymaker that sets 
prices $p(\cdot)$ for the items $\Theta$. Upon observing the prices, the buyer~$i$ 
maximizes their utility subject to the budget constraint. 
Their \emph{demand set} is the set of bundles that are optimal under the prices:
\[ D_i (p) := \argmax_{x_i \in L^\infty_+(\Theta)} \bigg\{ \int v_i x_i s \diff \t :  \int p x_i s\diff \t \leq b_i\bigg\} \; .\] 
Note that the demand set allows $x_i$ to take values greater than one.

Of course, due to the supply constraint, if prices are too low, there will be a supply shortage. On the other hand, if prices are too high, a surplus occurs. A competitive equilibrium is a set of prices and bundles such that all items are sold exactly at their supply (or have price zero).
We call such an equilibrium the \emph{limit LFM equilibrium} for the supply function $s$~\citep{gao2022infinite,liao2023fisher}.
\begin{defn}[Limit Linear Fisher Market Equilibrium]
    \label{def:LMF}
The limit equilibrium, denoted $\LFM (b,v,s, \Theta)$, is an allocation-price tuple $(x,p(\cd))$ such that the following holds.
\begin{enumerate}
    \item Supply feasibility and market clearance: $\sum_i x_i \leq 1$ and $\int p (1 - \sum_i x_i) s \diff \t= 0$. 
    \item Buyer optimality: $x_i \in D_i (p)$ all $i$.
\end{enumerate}
\end{defn}

Given the equilibrium quantities $(\xst, \pst)$, let $\usti = \int \vi s \xsti \diff \t$ be the buyer $i$ utility, and $\betasti = b_i / \usti$ be the buyer $i$ inverse bang-per-buck.
In an LFM, the equilibrium quantities $\pst,\betast,\ust$ are unique. 
Under twice differentiability (\nameref{as:twice_diff}; to be defined), the allocation $\xst$ is also unique.

Next we introduce the \emph{finite} LFM, which models the data we observe in a market.
Let $\gam = \{ \theta^1,\dots, \theta^t\}$ be $t$ i.i.d. samples from the supply distribution $s$, each with supply $1/t$.
See \cref{sec:defs_lfm_fppe} for a full definition. For a finite LFM, let the equilibrium per-buyer inverse bang-per-bucks be denoted by $\betagam$.

\begin{defn}[Finite LFM, informal] \label{def:observed_market}
    The finite LFM equilibrium, denoted $\oLFM$, is a limit LFM equilibrium where the item set $\Theta$ is the finite set of observed items $\gamma$. 
\end{defn}

It is well-known \citep{eisenberg1959consensus,gao2022infinite}
that the equilibrium inverse bang-per-buck $\betast$ in an limit (resp.\ finite) LFM uniquely solves
the population (resp.\ sample) dual EG program
\begin{align}
    \label{eq:pop_deg_lfm}
    \betast = \argmin_{\beta \in \Rnp}
     H(\beta) \; , \;
     \betagam = \argmin_{\b\in \Rnp} H_t(\b) \; .
\end{align}

The asymptotics of LFM were studied in \citet{liao2023fisher} under twice differentiability (\nameref{as:twice_diff}; to be defined). 
Let $\cH = \nabla \sq H (\betast)$.
They show $\sqrt t (\betagam - \betast)  \tod \cJ_\LFM$, where
\begin{align}
    \label{eq:def_cJLMF}
    \cJ_\LFM  = \cN \big(0, \cH\inv \E[\nabla F(\cd,\betast) \nabla F(\cd,\betast) \tp ] \cH \inv\big)  \;.
\end{align}

\subsection{First-Price Pacing Equilibrium (FPPE)}
The FPPE setting~\citep{conitzer2022pacing} models an economy that typically occurs on internet advertising platforms: the buyers (advertisers in the internet advertising setting) are subject to budget constraints, and must participate in a set of first-price auctions, each of which sells a single item.
Each buyer chooses a \emph{pacing multiplier} $\beta_i \in [0,1]$ that scales down their bids in the auctions, and submits bids of the form $\beta_i v_i(\theta)$ for each item $\theta$, with the goal of choosing $\beta_i$ such that there is \emph{no unnecessary pacing}, i.e. they spend their budget exactly, or they spend less than their budget but they do not scale down their bids.
In the FPPE model, all auctions occur simultaneously, and thus the buyers choose a single $\beta_i$ that determines their bid in all auctions.

\begin{defn}[Limit FPPE, \citet{gao2022infinite,conitzer2022pacing}]
    \label{def:limit_fppe}
    A limit FPPE, denoted $\FPPE(b,v,s, \Theta)$, is the unique tuple $(\beta, p(\cd)) \in [0,1]^n \times L^1_+ (\Theta)$ such that there exist $x_i : \Theta \to [0,1]$, $i\in[n]$ satisfying
    \begin{enumerate}[series = tobecont,itemjoin = \quad]
        \item (First-price) 
        Prices are determined by first-price auctions:
        for all items $\theta \in \Theta$, $p(\theta) = \max_i \betai v_i(\theta)$. 
        Only the highest bidders win:
        for all $i$ and~$\theta$, $x_i(\theta) > 0$ implies $\betai \vithe =\max_k \beta_k v_k (\theta)$ 
        \label{it:def:first_price}
        \item (Feasibility, market clearing)  
        Let $\pay_i = \int x_i(\theta) p(\theta) s(\theta)\diff \theta $.
        Buyers satisfy budgets:
        for all~$i$, $\pay_i  \leq b_i$. 
        There is no overselling: 
        for all $\theta$, $\sumiton x_i (\theta) \leq 1$.  
        \label{it:def:supply_and_budget}
        All items are fully allocated:
        for all~$\theta$, $p(\theta) > 0$ implies $ \sumiton x_i(\theta) = 1$.
        \item (No unnecessary pacing) For all $i$, $\pay_i < b_i$ implies $\betai = 1$. 
        \label{it:def:rev_max}
    \end{enumerate}
\end{defn}

FPPE is a hindsight and static solution concept for internet ad auctions. 
Suppose we know all the items that are going to show up on a platform.
FPPE describes how we could configure the $\betai$'s in a way that ensures that all buyers satisfy their budgets, while maintaining their expressed valuation ratios between items. 
Typically, the $\betai$'s are chosen by a pacing algorithm that is run by the platform.
FPPE has many nice properties, such as the fact that it is a competitive equilibrium, it is revenue-maximizing, revenue-monotone, shill-proof, has a unique set of prices, and so on~\citep{conitzer2022pacing}.
We refer readers to \citet{conitzer2022pacing,liao2023stat} for more context about the use of FPPE in internet ad auctions.

Let $\betast$ and $\pst$ be the unique FPPE equilibrium multipliers and prices. 
Revenue in the limit FPPE is 
$
        \REVst \defeq \int \pst(\theta) s(\theta)\diff \theta \; .
$
We let the leftover budget be denoted by  $\deltasti \defeq b_i - \pay_i$.
We say a buyer $i$ is \textit{degenerate} if $\betasti = 1$ and $\deltasti = 0$.

In FPPE the following regularity condition is important.
\begin{Assumption}[\textsf{\scriptsize{SCS}}] 
    There are no degenerate buyers, i.e., 
    \label{as:constraint_qualification}  $\betasti = 1$ implies  $\delta^*_i  > 0$. 
\end{Assumption}

This assumption is a strict complementary slackness condition since $\deltasti$ is the dual variable of $\betasti$ in the EG program introduced below.  
We will study the asymptotics of FPPE without \nameref{as:constraint_qualification}.
However, as we will see in \cref{sec:bs_with_scs}, condition \nameref{as:constraint_qualification} is helpful for bootstrap inference.

We let $\gam = \{ \theta^1,\dots, \theta^t\}$ be $t$ i.i.d.\ draws from $s$, each with supply $1/t $. They represent the items observed in an auction market.
The definition of a finite FPPE is parallel to that of a limit FPPE, except that we change the supply function to be a discrete distribution supported on the finite set $\gamma$.

\begin{defn}[Finite FPPE, informal]    
    A finite FPPE, $\oFPPE$, is a limit FPPE where the item set is the finite set of observed items $\gamma$. See \cref{sec:defs_lfm_fppe} for the full definition.
\end{defn}

It is well-known \citep{cole2017convex,conitzer2022pacing,gao2022infinite}
that $\beta$ in a limit (resp.\ finite) FPPE uniquely solves 
the population (resp.\ sample) dual EG program
\begin{align}
    \label{eq:pop_deg}
    \betast =  \argmin_{\beta \in (0, 1]^{n}}
     H(\beta) \;, \;
     \betagam = \argmin_{\b\in(0,1]^n} H_t(\b)  \;,
\end{align}
where the objectives $H$ and $H_t$ is the same as \cref{eq:pop_deg_lfm}. The difference between the LFM and FPPE convex programs is that for FPPE we impose the constraint $ (0,1]^n$.

The study of the asymptotics of FPPE was initiated by \citet{liao2023stat}. Let $\cJ_\FPPE$ be the limit distribution of $\sqrt t (\betagam - \betast) $, i.e., 
\begin{align}
    \sqrt t (\betagam - \betast) \tod \cJ_\FPPE \;.
    \label{eq:def_jfppe}
\end{align}
They show that, with the strict complementary slackness assumption 
\nameref{as:constraint_qualification}, the distribution $\cJ_\FPPE$ simplifies to
\begin{align}
    \cN \big(0, \,(P\cH P)\pinv \cov[\nabla F(\cd,\betast) ] (P \cH P) \pinv \big) \;,
    \label{eq:fppe_with_scs}
\end{align}
where $\cH = \nabla \sq H (\betast)$ and $P = \Diag(1(\betasti < 1))$.
Note this is a degenerate normal distribution supported on the hyperplane $\{h: h_i = 0, i \in I_{=} = I_{=>}\}$.

We will study the form of $\cJ_\FPPE$ assuming only twice differentiability (\nameref{as:twice_diff}) and not \nameref{as:constraint_qualification}. We will characterize $\cJ_\FPPE$ by a random quadratic program and provide several examples.
Thus, a contribution of our paper is to remove the strict complementarity slackness assumption and characterize the full landscape of FPPE asymptotics.

\subsection{Smoothness Assumptions}
The following assumption will be made throughout the paper, for both LFM and FPPE.
\begin{Assumption}[\textsf{\scriptsize{SMO}}]\label{as:twice_diff}
    The EG population objective $H(\cd)$ in \cref{eq:def_pop_eg} is twice continuously differentiable in a neighborhood of $\betast$. 
\end{Assumption}

\cref{as:twice_diff} implies that the Hessian $\cH = \nabla\sq H(\betast)$ is positive definite.
Here $\betast$ is interpreted as the equilibrium inverse bang-per-buck in a limit LFM, and  equilibrium pacing multipliers in a FPPE. See \citet{liao2023stat,liao2023fisher} for discussions of implications and concrete examples of \nameref{as:twice_diff} holding.

Our research goal can now be stated as
\begin{center}
    \noindent\fbox{%
    
        \parbox{.45\textwidth}{%
        \emph{Design bootstrap estimators of the distribution $\cJ_\LFM$ (resp.\ $\cJ_\FPPE$) given the observed market equilibrium $\oLFM  $ (resp.\ $\oFPPE $).}
        }%
    }
\end{center}

Inference on other quantities that are differentiable functions of $\betast$ can be achieved by the bootstrap delta method (\citet[Theorem 12.1]{kosorok2008introduction}, \citet[Theorem 3.10.11]{vaart2023empirical}).
For example, utilities $\usti = b_i / \betasti$ and the Nash social welfare $ \sumi b_i \log \usti = \sumi b_i \log (b_i / \betasti)$ are  smooth functions of $\betast$.
Revenue $ \int \max_i \{ \betasti \vithe\} s(\theta) \diff \t$ is also a smooth function of $\betast$.
For this reason, throughout the paper we will focus on inference of $\betast$, i.e., the utility prices in LFM and pacing multipliers in FPPE. 

\section{Bootstrapping Fisher Market Equilibrium}

In this section we let $\betagam$ be the observed utility prices in $\oLFM(b,v,1/t, \gamma)$, where $\gamma$ consists of $t$ i.i.d.\ draws from supply $s$. As mentioned previously, $\betagam = \argmin_{\Rnp} H_t(\b)$. The target distribution we want to estimate is $\cJ_\LFM$ in \cref{eq:def_cJLMF}.

\subsection{Exchangeable Bootstrap}
Define the exchangeable bootstrap by 
\begin{align}
    \betab_{\mathsf{ex, LFM}} = \argmin_{\beta \in \Rnp} \Pexbt F(\cd, \beta) \; .
    \label{eq:defexboot_lfm}
\end{align}
Compared with the convex program for LFM in \cref{eq:pop_deg_lfm}, the exchangeable bootstrap replaces $P_t$ with $\Pexbt$.
Exchangeable bootstrap is considered a smooth alternative to the traditional 
multinomial bootstrap (i.e. sampling with replacement) because it allows for a wider class of distributions of bootstrap weights \citep{praestgaard1993exchangeably}.
Concretely, we need the weights in the exchangeable bootstrap to satisfy the following conditions.
\begin{defn}[Exchangable bootstrap weights]
    \label{as:exboot}
    (1) The random vector $W=(W_{1}, \ldots, W_{t})\tp$ is exchangeable. 
    (2) $W_{ \tau} \geq 0$, and $\sum_{\tau=1}^t W_{ \tau}=t$.
    (3) $W_1$ has finite $(2+\epsilon)$ moment for some $\epsilon > 0$.
    (4) $\frac1t \sum_{\tau=1}^t\left(W_{ \tau}-1\right)^2 \toprob c^2>0$ as $t\to\infty$.
\end{defn}

Exchangeable bootstrap incorporates many popular forms of resampling as special cases such as the classical sampling with replacement, sampling without replacement, and normalized i.i.d.\ weights; see \cref{sec:exchangeable_bs_example}.

\begin{theorem}
    \label{thm:exboot_fm}
   
    $ \sqrt t (\betab_{\mathsf{ex, LFM}} - \betagam ) \topw  c\cdot \cJ_\LFM$ where 
    the constant $c$ is defined in \cref{as:exboot}.
    Proof in \cref{sec:proof:thm:exboot_fm}.
\end{theorem}

The proof of \cref{thm:exboot_fm} is complicated by the fact that the EG objective is nonsmooth due to the $\max$ operation in \cref{eq:defexboot_lfm}.
Establishing \cref{thm:exboot_fm} requires using the exchangeable bootstrap empirical process theory from \citet{praestgaard1993exchangeably} and \citet{wellner1996bootstrapping} to establish a form of stochastic differentiability (\cref{lm:stocdiff_boot} in appendix),
and applying the Taylor expansion-type analysis for nonsmooth objective functions from \citet{pollard1985new}.

In practice, approximate LMF equilibrium and bootstrap estimates suffice.
\cref{eq:defexboot_lfm} need not be solved exactly; error in the objective up to order $o_p(1/n)$ suffices, i.e., $\Pexbt F(\cd, \betab_{\mathsf{ex, LFM}} ) \leq  \min_\b  \Pexbt F(\cd, \beta)+o_p(1/n)$. And $\betagam$ only needs to be an approximate Fisher market equilibrium: $P_t F(\cd, \betagam ) \leq \min_\b P_t F(\cd, \b) + o_p(1/n)$.
The proof of \cref{thm:exboot_fm} can be extended to account for the extra error from approximate optimization.
In \cref{sec:nub_proxb_lfm} we briefly review two other valid bootstrap procedures, proximal bootstrap, and numerical bootstrap, and the consistency theory based on \citet{li2020numerical} and \citet{li2023proximal}.
Proximal bootstrap has the advantage of solving quadratic programs only.
However, those two methods converge at a rate slower than $1/\sqrt t$. In contrast, exchangeable bootstrap offers flexibility in choosing bootstrap weights, enjoys a $1/\sqrt t$ rate, and does not need parameter tuning.

\section{Bootstrapping FPPE}

In this section we let $\betagam$ be the pacing multiplier in $\oFPPE(b,v,1/t, \gamma)$, where $\gamma$ consists of $t$ i.i.d.\ draws from supply $s$. As mentioned previously, $\betagam = \argmin_{(0,1]^n} H_t(\b)$. The target distribution we want to estimate is $\cJ_\FPPE$ in \cref{eq:def_jfppe}, the limit distribution of $\sqrt t(\betagam  - \betast)$.

Bootstrapping FPPE is a significantly harder problem due to the presence of constraints in the EG program in \cref{eq:pop_deg}.
We investigate the full landscape of FPPE asymptotics, i.e., $\cJ_\FPPE$, in \cref{sec:fppe_limit_dist}. 
In \cref{sec:fail_multi_bs}, we show that the standard multinomial bootstrap fails to estimate $\cJ_\FPPE$ consistently. 
This also suggests that estimating $\cJ_\FPPE$ in full generality is difficult.
Because of this, we divide our study into an easier case and the harder case.
In the simpler case, we assume that all buyers exhaust their budget; for this case we show in \cref{sec:poorbuyerfppe} that the bootstrap methods from \citet{li2020numerical,li2023proximal} are valid.
A more realistic case is when some buyers do have leftover budgets. We design a bootstrap for this case in \cref{sec:bs_with_scs}, under an additional assumption of strong complementary slackness (\nameref{as:constraint_qualification}).
Finally, to complete the picture, we present a bootstrap-based confidence region for fully general FPPE in \cref{sec:generalfppe}.

\subsection{The Limit Distribution of General FPPE}
\label{sec:fppe_limit_dist}
The limit distribution of FPPE was studied in \citet{liao2023stat} under Assumption \nameref{as:constraint_qualification}. In this section, we characterize the full landscape of the asymptotics of FPPE without strict complementarity.
The convex program characterization in this section is a direct corollary of
noticing the connection between the results of \citet{shapiro1989asymptotic} and \citet{liao2023stat}. 
Concretely, Theorem 3.3 from \citet{shapiro1989asymptotic} established asymptotic distribution results for general constrained programs under equicontinuity conditions, and the results of \citet{liao2023stat} imply those equicontinuity conditions for the EG objective in \cref{eq:def:F}.
This is how we derive the convex program characterization of the asymptotics below. We then derive a new closed-form expression for the convex program, which allows us to analyze the asymptotic structure for several example.


To describe $\cJ_\FPPE$ we need to introduce a quadratic program.
Let $I_= = \{ i: \betasti = 1\}$ be the set of unpaced buyers and $\Ic = [n] \setminus I$. We further partition $I_=$ into 
\begin{align*}
    I_{=>} = \{i: \betasti =1, \deltasti > 0\}\;,
    \; I_{==} = \{i : \betasti = 1, \deltasti= 0\}\; .
\end{align*}

$I_{=>}$ is the set of buyers with strictly positive leftover budgets, whereas $I_{==}$ are the degenerate buyers.
From an optimization perspective, the set $I_{=>}$ corresponds to the strongly active constraints in the program \cref{eq:pop_deg}, whose corresponding Lagrange multipliers are strictly positive, while the set $I_{==}$ are the weakly active constraints, whose Lagrange multipliers are zero. With these notations, we note \nameref{as:constraint_qualification} is the same as $I_{==} = \emptyset$, and that the condition that all buyers exhaust their budgets is the same as $I_{=>} = \emptyset$.
Define $h:\Rn\to \Rn$,
\begin{align}
    \label{eq:fppe_asym_program}
    h(\xi) =  \argmin_{h \in \Rn;  h_i = 0, i \in I_{=>};  h_i \leq 0, i \in I_{==}} \| h + \cH\inv \xi \|_{\cH} \sq \; ,
\end{align}
where $\|a\|_\cH \sq = a\tp \cH a$.
The program \cref{eq:fppe_asym_program} can be interpreted as projecting the vector $-\cH\inv \xi$ onto the cone $\{h: h_i=0, i\in I_{=>}; h_i \leq 0, j\in I_{==} \}$ w.r.t.\ the norm $\|\cd\|_\cH$.
The function $h$ is continuous and positively homogeneous of degree 1, i.e., $h(t \xi) = t h(\xi)$ for $t > 0$, but not necessarily linear. When $I_{==} = \emptyset$, i.e., \nameref{as:constraint_qualification} holds, the function $h(\xi) = - (P\cH P)\pinv \xi$.

Combining Theorem 3.3 from \citet{shapiro1989asymptotic} with the equicontinuity results of \citet{liao2023stat}, we have that under the \nameref{as:twice_diff} assumption, 
$$\cJ_\FPPE \defeq h\Big( \cN \big(0, \cov[\nabla F(\cd,\betast)] \big) \Big)\;.$$

Note how different buyer types affect the support of $\cJ_\FPPE$.
The effects of $I_{==}$ and $I_{=>}$ are clear since they appear in the constraints. Buyers who do not win anything ($ \betasti = 1, \deltasti = b_i $) determine the support of $\cN( 0, \cov[\nabla F(\cd,\betast)])$, which is $
\{ g\in \Rn : g_i = 0  \text{ if $i$ does not win anything}\}$.

Below and in \cref{sec:detailsofexample} we study the form of $\cJ_\FPPE$ under some special cases by deriving closed-form expression of the quadratic program \cref{eq:fppe_asym_program}.

In the example below, we assume $I_{=>} = \emptyset$ for simplicity. 
Let 
$
    D = \Diag(\cH \inv ) ^{1/2}, \rho =  D\inv \cH\inv D\inv, Z = - D\inv \cH\inv G.
$
where $G\sim \cN(0, \cov[\nabla F(\cd,\betast)])$. Intuitively, $\rho$ is a normalized version of the inverse of the Hessian. Denote entries of $Z$ by $[Z_1,\dots, Z_n]\tp$. 
\begin{example}[The case with $|I_{==}| = 1$.]
    \label{ex:i01}
    Let $I_{=>} = \emptyset$, $I_{==} = \{ 1\}$ and $\Ic = \{2, \dots, n\}$.
    Then $\cJ_\FPPE = - \cH\inv G$ if $Z_1 < 0$, otherwise, if $Z_1 \geq 0$, then
    \begin{equation}
       \cJ_\FPPE  \defeq 
        D \begin{bmatrix}
            0 \\
            Z_{ 2}-\rho_{12} Z_{ 1} \\
            \vdots \\
            Z_{ n}-\rho_{1 n} Z_{ 1} 
        \end{bmatrix} \;. 
    \end{equation}
\end{example}

\cref{ex:i01} and \cref{eg:case_two_buyer} in appendix
illustrate an interesting phenomenon that the limit marginal distribution of the degenerate buyers (those with $\betasti = 1$ and $\deltasti = 0$)
is a distribution with some probability weight at $0$ and the rest on the negative reals.
This makes sense intuitively since in a finite sample, $\betagami - \betasti = \betagami - 1$ is always negative for $i \in I_{==}$.
Another feature of $\cJ_\FPPE$ is that the limit distribution of $\sqrt{t}(\betagam_i - 1)$ is degenerate (a point mass at zero) if $i\in I_{=>}$. This also implies $\betagami - 1 = o_p(\frac{1}{\sqrt t})$ if $i \in I_{=>}$.

\subsection{Failure of Multinomial Bootstrap for FPPE}
\label{sec:fail_multi_bs}
As described in \citet{andrews2000inconsistency}, standard multinomial bootstrap might fail in constrained programs.
In this section, we show that this is the case for FPPE.

Consider a one-buyer FPPE. Let $b_1 = 1$, $\E[v_1] = \int v_1 s \diff \theta = 1$ and $s$ is the supply (a probability density). 
Let $\gamma = \{\thetau\}_\tau$
be i.i.d.\ draws from $s$. 
Let $\betagam$ be the pacing multiplier in $\oFPPE(b,v,1/t,\gamma)$ and $\betast$ be that in $\FPPE(b,v,s,\Theta)$.

Given the observed items, let $\{ \theta^{\tau,\boot}\}_\tau$ be the resampled items (with replacement).
For this instance,
the bootstrapped FPPE with standard multinomial weights is 
\begin{align}
    \label{eq:bsexample_eg}
    \beta^{\boot} \defeq \argmin_{\beta_1 \in (0,1] }  \frac1t \sumtau  \beta_1 v_1(\theta^{\tau,\boot}) - b_1 \log \beta_1 \;.
\end{align}

\begin{theorem}[Failure of Multinomial Bootstrap]
    \label{thm:failbs}
    The limit conditional distribution of $\sqrt t (\betab - \betagam)$ is not equal to the limit distribution of $\sqrt t (\betagam - \betast)$. 
    Proof in \cref{sec:proof:thm:failbs}.
\end{theorem}

In fact, using the same argument in \citet{abrevaya2005bootstrap}, we can show that in the above example, $\sqrt t(\betagam - \betast) \tod \min\{Z_1, 0\}$, while 
$ \sqrt t (\betab - \betagam) \topw \min \{Z_1 + Z_2, 0\} - \min\{ Z_1, 0 \}$ 
where $Z_1,Z_2$ are independent copies of $\cN(0, \var(v_1))$.

\subsection{Bootstrapping FPPE under Full Budget Exhaustion}
\label{sec:poorbuyerfppe}
If FPPE has the additional structure that all buyers exhaust their budgets, i.e., $I_{=>} = \emptyset$, we can apply the 
numerical bootstrap \citep{li2020numerical} and the proximal bootstrap \citep{li2023proximal}. 
Equivalently, it requires the population EG in \cref{eq:pop_deg}
does not have strongly active constraints, $\nabla H(\betast) = 0 $, and the unconstrained optimum
coincides with the constrained optimum.

Under this additional structure, \cref{eq:fppe_asym_program} becomes 
\begin{align}
    \cJ_\FPPE = \argmin_{h : h_i \leq 0, i \in I_{==}} \| h + \cH \inv G\|_{\cH}\sq
    \notag
    \\
    = \argmin_{h: h_i \leq 0, i \in I_{==}} h\tp G + \tfrac12 h \tp \cH h
    \label{eq:expaned_stoch_prog}
\end{align}
and $G \sim \cN(0, \E[\nabla F(\cd, \betast)\nabla F(\cd, \betast)\tp])$.
So $\cJ_\FPPE$ is a distribution supported on the cone $\{ h \in \Rn: h_i \leq 0 \text{ if } i \in I_{==} \}$, with some probability mass distributed on the faces of the cone.

To obtain numerical bootstrap and proximal bootstrap estimates, 
we require a smoothing parameter $\epsilon_t \downarrow 0$ such that $\epsilon_t \sqrt{t} \to \infty$.
Then, to get $\betab_{\mathsf{nu,FPPE}}$ we solve
\begin{align}
    \argmin_{\b \in (0,1]^n} \frac1t \sumtau (1 + \ept \sqrt t (W_\tau -1)) F(\thetau,\b) \;,
     \label{eq:def_nu_bs}
\end{align}
and to get $\betab_\mathsf{pr, FPPE} $ we solve
\begin{align}
    {\argmin_{\beta \in [0,1]^n}   \epsilon_t (G^b)\tp (\beta - \betagam) + \tfrac12 \| \beta - \betagam \| _ {\hat \cH } \sq } \;,
    \label{eq:def_prox_bs}
\end{align}
where
\begin{align}
    \label{eq:def_Gb_H}
    G^b \defeq \sqrt t (\Pbt - \Pt) D_F(\cd, \betagam) \;, \;
    \hat \cH \defeq (\hat \cH_{k,\ell})_{k,\ell} \;.
\end{align}
Here
$D_F(\cd,\betagam)$ is a deterministic element in the subdifferential $\partial _{\b}F (\cd, \betagam) $
\footnote{
    We avoid writing $\nabla F(\cd, \betagam)$ 
    because in a finite FPPE there could be ties. And when ties happen for an item $\theta$, EG objective $\b\mapsto F(\t, \b)$ is not differentiable at $\betagam$.}.
The term $G^b$ estimates the Gaussian random variable $G$ in \cref{eq:expaned_stoch_prog}.
The numerical difference estimator is $\hat \cH_{k,\ell} = (\tdifft H_t)(\betagam)$, where 
    \begin{align*}
(\hat \nabla \sq _{k\ell, \eta} g)(\cdot) = [
        g(\cdot+ \eta e_k + \eta e_\ell) 
        -g(\cdot- \eta e_k + \eta e_\ell)\\
        -g(\cdot+ \eta e_k - \eta e_\ell)
        +g(\cdot- \eta e_k - \eta e_\ell) ]/(4\eta\sq),
\end{align*}
and $H_t$ is the finite-sample EG objective in \cref{eq:def_pop_eg}. In practice, both \cref{eq:def_nu_bs,eq:def_prox_bs} only need to be solved approximately with error in the objective up to $o_p(\ept \sq)$.
The proximal bootstrap in \cref{eq:def_prox_bs} is a bootstrap analogue of the distribution in \cref{eq:expaned_stoch_prog}.

The following theorem shows that in the budget-exhaustion case, the numerical bootstrap and proximal bootstrap converge to the correct limit distribution.
The proofs can be found in \cref{sec:proof:thm:nuboot_fm,sec:proof:thm:prboot_fm}.
\begin{theorem} 
    \label{thm:fppepoor}
If all buyers exhaust their budgets ($I_{=>} = \emptyset$), $\dtt = o(1) $, $\sqrt t\dtt  \to \infty$, then
    \begin{enumthmresult}
        \item  
        \label{thm:nuboot_fppepoor}
        $
        \epsilon_t\inv (\betab_{\mathsf{nu, FPPE}} - \betagam) \topw \cJ_\FPPE
        $ . 
        \item
        \label{thm:prboot_fppepoor}
        If $ \hat \cH \toprob \cH$, then
        $\epsilon_t 
        \inv (\betab_{\mathsf{pr, FPPE}} - \betagam ) \topw \cJ_\FPPE$ . 
    \end{enumthmresult}
\end{theorem}

The proof proceeds by verifying conditions in \citet{li2020numerical} and \citet{li2023proximal}. Stochastic equicontinuity of certain processes is verified using results from \citet{liao2023stat}.
\begin{figure*}[h!]
    \centering
    \includegraphics[scale=.26]{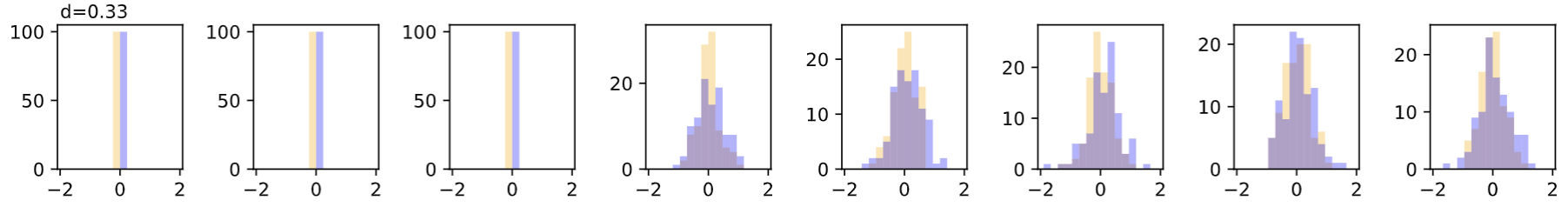}
    \caption{Bootstrap vs finite-sample distribution of an 8-buyer 1000-item FPPE.
    Values are i.i.d.\ uniformly distributed, and budgets are generated randomly in a way that the first three buyers have leftover budgets. Displayed are histograms of $\beta_1,\dots, \beta_8$.
    Purple: 100 samples of $\ept\inv (\betab-\betagam)$ according to \cref{eq:adboot_fppe} given one FPPE. Yellow: 100 samples of $\sqrt t (\betagam - \betast)$.
    Bootstrap distribution is very similar to FPPE distribution. The similarity is significant, because to obtain the distributions of FPPE, we need to observe multiple market equilibria, to which we usually do not have access.
    The bootstrap distribution, on the other hand, is generated based on just one finite FPPE. }
    \label{fig:slice_of_plot}
\end{figure*}

\subsection{Bootstrapping FPPE without Degenerate Buyers} 
\label{sec:bs_with_scs}

In real-world auction markets such as those at internet companies, some fraction of buyers do have leftover budgets~\citep{conitzer2022multiplicative}.
In this section, we give a bootstrap estimate of $\cJ_\FPPE$ under \nameref{as:constraint_qualification},
in which we allow users to have positive leftover budgets, but rule out degenerate buyers. 
Condition~\nameref{as:constraint_qualification} is equivalent to 
requiring that in the population EG in \cref{eq:pop_deg} there is no weakly active constraints (those whose Lagrangian multipliers are zero).
Condition~\nameref{as:constraint_qualification},
equivalent to $I_{==} = \emptyset$,
is realistic because degenerate buyers are a measure-zero edge case.
Note now $\cJ_\FPPE$ is a degenerate normal distribution supported on the hyperplane $\{h: h_i = 0, i \in I_{=} = I_{=>}\}$.

Choose two vanishing sequences $\dwt$ and $\dtt$.
Define the \emph{estimated} unpaced buyers $\Ihatp= \{i: \betagami > 1 - \dwt\}$
and the reduced feasible set $\hat B \defeq \{ \beta \in [0,1]^n: \beta_i = 1 \text{ for } i \in \hat I_{=} \}$. 
The proposed bootstrap estimator is  
\begin{align}
    \label{eq:adboot_fppe}
    \beta^{\boot} 
    & \defeq \argmin_{\beta \in \hat B}  \dtt (G^b)\tp (\beta - \betagam) + \tfrac12 \|\betagam- \beta\|_{\hat \cH} \sq  \;,
\end{align}
where
$G^\boot $ and $\hat \cH$ are defined in \cref{eq:def_Gb_H}.
The estimator has a nice geometric interpretation: we add certain appropriate noise to $\betagam$ and then project back to the reduced feasible set $\hat B$.
We call $\eps_t$ the bootstrap stepsize, whose effect is investigated in \cref{sec:exp}.
\begin{theorem}
    \label{thm:adaptive_boot}
    Let \nameref{as:constraint_qualification} hold in FPPE ($I_{==}=\emptyset$).
    Let $\dwt$, $\dtt = o(1) $, $\sqrt t\dtt  \to \infty$, $ \sqrt t \dwt \to c \in (0, \infty]$.
    If $ \hat \cH \toprob \cH$,
    then
$    \dtt ^{-1} ( \beta^{\boot} - \betagam) \topw \cJ_\FPPE$. Proof in \cref{sec:proof:thm:adaptive_boot}.
\end{theorem}

The estimator in \cref{eq:adboot_fppe} is proposed following ideas from \citet{li2023proximal,cattaneo2020bootstrap}, where the bootstrap is in fact approximating the 
random quadratic program \cref{eq:fppe_asym_program}.
Many existing works \citep{geyer1994asymptotics,li2020numerical,li2023proximal} require that strongly active constraints do not occur, and are thus not applicable for FPPE with buyers who have leftover budgets.
As with proximal bootstrap, our approach requires solving quadratic programs only. 

We briefly remark on the techniques used to prove \cref{thm:adaptive_boot}. 
We combine the 
theory of weak convergence \citep{vaart2023empirical} from statistics 
and epi-convergence theory~\citep{rockafellar1970convex} from optimization.
The reason is that weak convergence is a powerful tool to study asymptotics of statistical functionals, such as the $\argmin$ function, and epi-convergence is designed for studying constrained programs.
Such an approach dates back to \citet{geyer1994asymptotics} and \citet{molchanov2005theory}, and 
more recently was used by \citet{parker2019asymptotic} for constrained quantile regression, and \citet{li2020numerical} and \citet{li2023proximal} in the context of bootstrap.

Both proximal bootstrap (\cref{eq:def_prox_bs}) and 
our proposed bootstrap (\cref{eq:adboot_fppe})
require a numerical difference estimate of the Hessian (\cref{eq:def_Gb_H}). We provide a theorem to guide the choice of differencing stepsize.
\begin{theorem}[Hessian estimation, informal] 
    Consider the finite difference estimate defined in \cref{eq:def_Gb_H} with differencing stepsize $\eta_t = o(1)$ and $\eta_t\sqrt t \to \infty$. 
    Under regularity conditions,  $\hat \cH_{k\ell} - \cH_{k\ell} \asymp \eta_t^2 + \frac{1}{\eta_t \sqrt t} + \text{higher order terms}$.
    \label{thm:hessian_est}
    Proof in \cref{sec:thm:hessian_est}.
  \end{theorem}
  
  By setting $\eta_t\sq = 1/(\eta_t \sqrt t)$ we obtain the optimal choice $\eta_t \asymp t^{-1/6}$.
  The proof of \cref{thm:hessian_est} uses empirical process theory to handle the nonsmoothness of the EG objective.
  
\subsection{Confidence Regions for General FPPE}
\label{sec:generalfppe}

In \cref{sec:poorbuyerfppe,sec:bs_with_scs} we assumed either $I_{==}$ or $I_{=>}$ to be empty sets. 
Now we discuss bootstrap inference without such assumptions.
We can construct a confidence region for $\betast$ using bootstrap test inversion. Suppose we have a scalar statistic $T(\betast, \deltast, \theta^1,\dots,\theta^t )$, and  an upper bound estimate $c \in \R$ of the $(1-\alpha)$-quantile of its limit distribution.
Then the region $\{ (\beta,\delta): T(\beta, \delta, \theta^1,\dots,\theta^t ) \leq c \}$ is an asymptotically-valid confidence region
for $(\betast,\deltast)$.

First, we introduce a statistic based on the Lagrangian of the EG program.
The idea of using the Lagrangian or Karush-Kuhn-Tucker (KKT) system for inference in constrained programs also appears in \citet{li2023proximal,hsieh2022inference}.
Consider the sample Lagrangian
$L_t(\beta, \delta) = H_t(\beta) - \delta \tp (1_n - \b) $ for $\beta \in (0,1]^n$ and $0\leq \delta \leq b$. 
Define the statistic for some $\kappa \in (0, \infty]$:
\begin{align}
     T^\gam(\beta,\delta)   = - \inf_{ \beta' \in \beta + \frac{1}{\sqrt t }\B_\kappa} t \big(L_t(\beta', \delta) - L_t (\beta, \delta) \big) \;,
    \label{eq:def_Tgam}
\end{align}
where $\B_\kappa = \{ h\in \Rn: \|h\nmt \leq \kappa \}$.
The statistic $T^\gam$ finds the local minimum value of the Lagrangian over a $\frac1{\sqrt t}$-neighborhood of $\beta$.

Next, we introduce the bootstrap $T^b$ to estimate an ``upper bound'' of the distribution of $T^\gamma(\betast,\deltast)$.
\begin{align*}
    T^b 
    =  - \inf_{\beta \in \Rnp}(\dtt (G^b)\tp (\beta - \betagam) + \tfrac12 \|\betagam- \beta\|_{\hat \cH} \sq)  / (\dtt)\sq\;.
\end{align*}
The quadratic function of $\beta$ inside $\inf$ aims to estimate a quadratic expansion of the Lagrangian at $(\betast,\deltast)$.

To see that $T^b$ is an upper bound of $T^\gamma(\betast, \deltast)$, it turns out that 
$T^\gamma (\betast, \deltast) \tod - \inf_{h \in \B_\kappa} G \tp h + \frac12 h\tp \cH h $, while $ T^b \topw - \inf_{h \in \Rn} G \tp h + \frac12 h\tp \cH h$ where $G\sim \cN(0, \cov(\nabla F(\cd, \betast))) $.
For the same realization of $G$, the limit distribution of $T^b$ is greater than or equal to that of $T^\gamma(\betast,\deltast)$.

Now we are ready to introduce the confidence region.
Given a threshold value $c$, the statistic $T^\gamma$ induces the region 
\begin{align*}
    C^\gamma ( c )  & = \{ (\beta,\delta) : T^\gam(\beta,\delta) \leq c  \} 
    \\
    & \cap \{(1_n - \beta) \tp \delta = 0, 0 \leq \b \leq 1_n, 0\leq \delta \leq b \} \;.
\end{align*}
Let $c^b_{1-\alpha}$ be the conditional $(1-\alpha)$-quantile of $T^b$, i.e., $c^b_{1-\alpha} = \inf \{ x: \P(T^b \leq x | \{ \thetau\} ) \geq 1-\alpha\}$. Then a confidence region for $(\betast, \deltast)$ is $C ^\gamma(c^b_{1-\alpha})$.

\begin{theorem}\label{thm:generalFPPE_coverage}

    Suppose 
    $\dtt = o(1)$,
    $\dtt \sqrt t \to \infty$, $\hat \cH \toprob \cH$.
    Let $T^ \infty$ be the limit distribution of $T^\gam (\betast, \deltast)$. 
If the CDF of $T^\infty$ is continuous at the $(1-\alpha)$-th quantile of $T^\infty$, then 
$\liminf_{t\to\infty} \P( ( \betast,\deltast) \in C^\gamma(c^b_{1-\alpha})) \geq 1-\alpha$.
Proof in \cref{sec:proof:thm:generalFPPE_coverage}.
\end{theorem}

The condition on the continuity of the CDF is mild and commonly seen in the literature \citep{beran1984}. 
The cost that comes with the general applicability of the confidence region $C^\gam(c^b_{1-\alpha})$ is computational.
To decide whether a point $(\b,\delta)$ is in the region one solves the optimization problem in \cref{eq:def_Tgam}.

\section{Experiments}

We now conduct experiments to investigate the 
performance of the bootstrap estimator \cref{eq:adboot_fppe}
in FPPE with \nameref{as:constraint_qualification} conditions.
We aim to (1) verify that the bootstrap produces a consistent estimate of the FPPE asymptotic distribution, and (2) study the effect on the bootstrap of the stepsize parameter $\ept$ and market parameters, such as the number of items, number of buyers, proportion of budget-constrained buyers, and the value distributions. 

\textbf{Synthetic experiments.}
In \cref{sec:exp_sim} 
we consider an ideal scenario where buyers' values are i.i.d.\ draws from some distribution, i.e., $v_1,\dots, v_n \sim_{iid} F_v$. 
To assess the effect of the tail of the value distributions, we take $F_v$ to be a uniform, exponential, or truncated normal distribution.
We visualize and compare two setups:
\emph{1) true resampling,}
where the finite-sample distribution of $\sqrt t (\betagam - \betast)$, obtained by repeatedly drawing independent FPPE instances, and 
\emph{2) bootstrap:} $\ept\inv (\betab - \betagam)$ as defined in \cref{eq:adboot_fppe}, obtained by bootstrapping only one FPPE instance. 
We also vary the bootstrap stepsize $\epsilon_t$. 
Experiments confirm that our bootstrap \cref{eq:adboot_fppe} is consistent, fairly robust under a wide range of market parameters when bootstrap stepsize is chosen appropriately.

\textbf{Semi-real experiments.}
In \cref{sec:exp_ipinyou} we construct realistic instances from real-world auction markets based on the iPinYou dataset \citep{liao2014ipinyou}.
The dataset contains raw log data of the bid, impression, click, and conversion history on the iPinYou platform.
From the dataset we estimate the click-through rate of 
impressions using logistic regression and simulate realistic advertisers' values by perturbing the regression coefficients.
We treat the sum of \pmrs as the target parameter and use percentiles of the bootstrap estimates based on \cref{eq:adboot_fppe} to construct confidence intervals.
We assess the effect on the coverage rate of the number of items, number of advertisers, the bootstrap stepsize $\ept$, and the proportion of unpaced buyers.
These experiments show that our bootstrap is  suitable for realistic auction markets.

\section{Future Directions}

A bootstrap theory for FPPE without regularity conditions on either the buyers (e.g. \cref{as:constraint_qualification}) or the CDF assumption in \cref{thm:generalFPPE_coverage} would be desirable. 
However, we suspect that this will be a difficult task, since bootstrapping completely general constrained convex programs remains an open problem. 
Secondly, we saw in our experiments that Hessian estimation is important for the performance of our bootstrap methods. Thus, a better understanding of how to perform Hessian estimation for the best performance on real-world problems would be useful.
In practice it would also be highly desirable to have a bootstrap theory that has some form of guarantees under nonstationary input data.

\section*{Acknowledgements}
This research was supported by the Office of Naval Research awards N00014-22-1-2530 and N00014-23-1-2374, and the National Science Foundation awards IIS-2147361 and IIS-2238960.

\section*{Impact Statement}
This paper presents work whose goal is to advance the field of Machine Learning. There are many potential societal consequences of our work, none which we feel must be specifically highlighted here.

\newpage

%% file: tech_lemma.tex
\section{Omitted Main Text}
\subsection{Notations}

Let $A = [A_1  ; \dots; A_n ]$ denote the matrix constructed by stacking $A_i$ from top to bottom. Vectors are column vectors by default. 
For a matrix $H$, a vector $G$, an index set $I \subset [n]$, we let $H_{II} = (H_{ij})_{i\in I, j\in I}$ to denote the $|I|\times |I|$ matrix consisting of entries in $H$, and $G_I$ be the subvector with entries indexed by $I$. We let $\Rnp = [0,\infty )^n$. 
Furthermore, we let $A\pinv $ be the Moore-Penrose pseudo inverse of a matrix $A$.
Denote by $e_j$  the $j$-th unit vector.

For a measurable space $(\Theta, \diff \theta)$, we let $L^p$ (and $L^p_+$, resp.) denote the set of (nonnegative, resp.) $L^p$ functions on $\Theta$ w.r.t\ the integrating measure $\diff \theta $ for any $p\in [1, \infty]$ (including $p=\infty$). 
We treat all functions that agree on all but a measure-zero set as the same.
For a sequence of random variables $\{X_n\}$, we say $X_n = O_p(1)$ if for any $\epsilon > 0$ there exists a finite $M_\epsilon$ and a finite $N_\epsilon$ such that $\P(|X_n| > M_\epsilon) < \epsilon$ for all $n\geq N_\epsilon$. We say $X_n = o_p(1)$ if $X_n$ converges to zero in probability. 
\begin{tabular}{r|l}
    \textbf{Symbol} & \textbf{Meaning} \\
    \hline
    $\tod$ & convergence in distribution of random vectors \\
    $\rightsquigarrow$ & weak convergence in metric space\\
    $\topw$ & weak conditional convergence in metric space \\
    $b, b_i$ & budgets \\
    $\betast, \betagam, \betab$ & pacing multipliers \\
    $e_i$ & the $i$-th basis vector \\
    $\epsilon_t$ & stepsize parameter in numerical bootstrap and proximal bootstrap and the proposed bootstrap \cref{eq:adboot_fppe} \\
    $\eta_t$ & differencing stepsize in finite-difference estimator of Hessian\\
    $\deltast, \deltasti$ &  leftover budget \\
    $\delta_t$ & constraint slackness in the proposed bootstrap \cref{eq:adboot_fppe} \\
    $D_i(p)$ & demand set in a Fisher market \\
    $F$ and $D_F$ & $F$ is the EG objective defined in \cref{eq:def:F}, and $D_F$ a deterministic selection of subgradients\\
    $\cov(\nabla F(\cd,\betast))$ & $\E[( \nabla F(\cd,\betast) - \E[\nabla F(\cd,\betast)])(\nabla F(\cd,\betast) - \E[\nabla F(\cd,\betast)])\tp]$\\
    $\gamma$ & observed item set in LFM and FPPE \\
    $G, G^b$ & a normal random variable $\cN(0, \cov(\nabla F(\cd,\betast)))$ and its bootstrap estimate \\
    $h$ & the quadratic program in \cref{eq:fppe_asym_program} \\
    $\cH$, $\widehat {\cH}$ & the Hessian matrix of $H$ at $\betast$, and its finite-difference estimator \\
    $H, H_t$ &  population and sample dual EG objective \\
    $I$, $\Ic$& The set of unpaced ($\beta_i^*=1$) and paced buyers ($\beta_i^* < 1$) buyers, respectively  \\
    $I_{==}$, $I_{=>}$& The set of unpaced buyers with $\delta_i^*=0$ and $\delta_i^* >0$, respectively  \\
    $\cJ_\LFM, \cJ_\FPPE$ & limit distributions of interest, defined in \cref{eq:def_cJLMF,eq:def_jfppe} \\
    $ \ell^\infty (K)$ & the space of bounded functions $f: K \to \R$ \\
    $p$ & price function in Fisher market and FPPE \\ 
    $P$ & matrix whose diagonal is $\indi(\betasti < 1)$, $i \in [n]$ \\
    $P_t, \Pexbt, P^b_t$ & expectation operators for the empirical distribution, exchangeable bootstrap distribution,\\& and the classical multinomial bootstrap distribution \\ 
    $\ust,\usti$ & equilibrium utility values in LFM and FPPE \\
    $s(\cdot)$ & supply function (a probability density) \\
    $v, v_i, v_i(\t)$ &  valuation functions\\
    $\xst, \xsti$ & equilibrium allocations in LFM and FPPE
    \end{tabular}

\subsection{Related Work} \label{sec:relatedwork}

    \textbf{Statistical Inference in Equilibrium Models.}
    \citet{liao2023fisher,liao2023stat}
    study statistical properties of LFM and FPPE, respectively.
    \citet{Wager2021,munro2021treatment,sahoo2022policy}
    take a mean-field game modeling approach and perform policy learning with a gradient descent method.
    \citet{johari2022experimental}
    study a Markov chain model of two-sided platform and investigate the effect of bias under different market balance condition.
    \citet{munro2023causal}
    considers global treatment effects in a market where the allocation mechanism exhibits certain structures. 
    Different from these work, this paper focuses on 
    estimating the asymptotic distribution of the market equilibrium,
    uses bootstrap to conduct inference and develops its statistical theory in the specific models of LFM and FPPE.

    \textbf{Bootstrapping $M$-estimators/mathematical programs.}
    There is a line of research on bootstrapping $M$-estimators 
    \citep{lahiri1992bootstrapping,gine1992bootstrap,wellner1996bootstrapping,bose2001generalised,lee2012general,patra2018consistent,cattaneo2020bootstrap}.
    For constrained $M$-estimators, one needs to be cautious about bootstrap procedures since they could produce inconsistent estimates of the target distribution~\citep{andrews2000inconsistency}.
    Bootstrapping constrained estimators is studied in \citet{li2023proximal} and \citet{li2020numerical}; in these works it is assumed that there are no strongly active constraints.
    In the FPPE setting strongly active constraints do occur, and we use epi-convergence theory to remedy this.

\subsection{Examples of Exchangeable Bootstrap}
\label{sec:exchangeable_bs_example}

\begin{example}
    The multinomial bootstrap corresponds to sampling with replacement. It satisfies \cref{as:exboot} with the constant $c^2 = 1$.
\end{example}

\begin{example}
    Sample without replacement.
    Let $h = \lfloor \alpha t \rfloor$ be the number of samples not chosen for some $\alpha \in (0,1)$. Concretely, let $w_\tau = \frac{t}{t-h}$ for $1\leq \tau \leq t-h$ and $0$ otherwise.
    Then $W$ is the vector of $(w_1,\dots, w_t)$ ordered at random independent of data.
    \cref{as:exboot} is satisfied with the constant $c^2 = \alpha / (1-\alpha)$.
\end{example}

\begin{example} I.i.d.\ weights.
    Let $w_1, \dots, w_t$ be i.i.d.\ draws from some distribution with finite $(2+\epsilon)$ moment, and $\bar w = \frac1t \sumtau w_\tau$. Define the bootstrap weights $W_\tau = w_\tau / \bar w$.
    \cref{as:exboot} is satisfied with the constant $c^2 = \var(w_1) / (\E[w_1])^2$.
\end{example}
For more examples of exchangeable bootstrap weights we refer readers to \citet{praestgaard1993exchangeably} and \citet{cheng2015moment}.
The wide range of bootstrap weights allowed by \cref{as:exboot} provides flexibility for practical application.

\subsection{Definition of Finite LFM and FPPE}
\label{sec:defs_lfm_fppe}
Here we give a formal definition of finite LFM and FPPE.
Let $\vitau = \vithetau$ be the valuation for the $\tau$'th sampled item.

\begin{defn}[Finite LFM] \label{def:observed_market}
    The finite observed LFM, denoted $\oLFM(b,v,\sigma, \gamma)$, is a allocation-price tuple $({x}, p) \in \R^{t\times n}_+ \times \Rnp $ such that the following hold:
\begin{enumerate}
    \item Supply feasibility and market clearance: $\sumi\xtaui \leq 1 $ and  $\sum_\tau \ptau (1- \sumi \xitau)  = 0$.

    \item Buyer optimality: $x_i \in D_i (p) = \argmax_{x_i} \{ \sum_\tau \xitau \vitau : \sigma \sum_\tau \xitau \ptau \leq b_i , 0 \leq \xitau \leq 1 \}$, the demand set given the prices. 
\end{enumerate}
\label{defn:observed_market}
\end{defn}
Suppose we have a finite LFM equilibrium $(x,p)=\oLFM(b,v, \sigma = 1/t, \gamma)$.
Then  $u^\gam_i = \sigma \sumtau \xitau \vitau $ is the utility of buyer $i$ in equilibrium, and $\betagami = b_i / \ugami$ is the utility price of buyer $i$.

\begin{defn}[Finite FPPE, \citet{conitzer2022pacing}]
    \label{def:observed_fppe}
   The finite observed FPPE, $\oFPPE(b,v, \sigma , \gamma)$, is the unique tuple
    $(\beta,p) \in [0,1]^n \times \R^t_+ $ 
    such that there exists $x_i^\tau \in [0,1]$ satisfying:
\begin{enumerate}
    \item 
    (First-price) For all $\tau$, $\ptau = \max_i \betai \vitau$. For all $i$ and $\tau$, $\xitau > 0$ implies $\betai \vitau =\max_k \beta_k v_k^\tau$. 
    \item
    (Supply and budget feasible)  For all $i$, $ \sigma \sum_\tau \xitau \ptau  \leq b_i$. For all $\tau$, $\sumi\xitau \leq 1$.  
    \item
    (Market clearing)  For all $\tau$, $\ptau > 0$ implies $ \sumi\xitau = 1 $.
    \item
    (No unnecessary pacing) For all $i$, $ \sigma \sum_\tau \xitau \ptau  < b_i$ implies $\betai = 1$.
\end{enumerate}
\end{defn}

\subsection{Numerical Bootstrap and Proximal Bootstrap for LFM}
\label{sec:nub_proxb_lfm}

We briefly review two valid bootstrap procedures and the consistency theory based on \citet{li2020numerical,li2023proximal}.
In this section we only consider multinomial bootstrap weights.

Given a sequence $\epsilon_t$ of positive numbers converging zero, the numerical bootstrap estimator is defined as
\begin{align*}
    \betab_{\mathsf{nu, LFM}} 
    & = \argmin_{\beta \in \Rnp} (\Pt + \epsilon_t \sqrt{t}(\Pbt - \Pt)) F(\cdot,\beta)
    \\
    & = \argmin_{\b \in \Rnp} \frac1t \sumtau (1 + \ept \sqrt t (W_\tau -1)) F(\thetau,\b)
\end{align*}
\begin{restatable}{theorem}{nubootlfm}
    \label{thm:nuboot_fm}
    Let $\epsilon_t = o(1)$ and $\epsilon_t\sqrt{t}\to \infty$.
    Then 
    $\epsilon_t\inv(\betab_{\mathsf{nu, LFM}} - \betagam ) \topw \cJ_\LFM$. 
\end{restatable}
The proof is in \cref{sec:proof:thm:nuboot_fm}.

Numerical bootstrap does not offer computational benefits, since it requires solving EG programs. However, as we see in \cref{sec:poorbuyerfppe} the idea of proximal bootstrap extends to a special case of FPPE where all buyers spend their budgets.
The regular multinomial bootstrap is recovered by setting $\ept = 1/\sqrt t$.

To describe proximal bootstrap, we define 
\begin{align*}
    G^b \defeq \sqrt t (\Pbt - \Pt) D_F(\cd, \betagam)
\end{align*}
and the numerical difference estimator of the Hessian matrix
$\hat \cH$, whose $(k,\ell)$-th entry is $\hat \cH_{k,\ell} = (\tdifft H_t)(\betagam)$, where 
    $(\tdiff g)(\cdot) = [
        g(\cdot+ \epsilon e_k + \epsilon e_\ell) 
        -g(\cdot- \epsilon e_k + \epsilon e_\ell)
        -g(\cdot+ \epsilon e_k - \epsilon e_\ell)
        +g(\cdot- \epsilon e_k - \epsilon e_\ell) ]/(4\epsilon\sq)$.
And $D_F(\cd, \betagam)$ is a deterministic element in $\partial F (\cd, \betagam)$.

The proximal bootstrap estimator is defined as 
\begin{align}
    \label{eq:defnuboot_lfm}
    \betab_{\mathsf{pr, LFM}} \defeq \argmin_{\beta \in \Rnp} \{ \epsilon_t (G^b) \tp (\beta - \betagam) + \frac12 \|\betagam - \beta\|_{\hat \cH} \sq  \}
\end{align}

\begin{restatable}{theorem}{prbootfm}
    \label{thm:prboot_fm}
    Let $\epsilon_t\sqrt{t}\to \infty$ and $\ept \downarrow 0$.
    Then 
    $\epsilon_t\inv(\betab_{\mathsf{pr, LFM}} - \betagam ) \topw \cJ_\LFM$. 
\end{restatable}
The proof is in \cref{sec:proof:thm:prboot_fm}.

Proximal bootstrap is clearly computationally cheap since it only requires solving an unconstrained convex quadratic program (as opposed to the exponential cone program for EG). On the other hand, the numerical bootstrap requires estimation of the Hessian matrix. See \cref{thm:hessian_est} for a discussion of stepsize selection when using finite difference methods to estimate the Hessian.

\subsection{Examples of FPPE limit distributions}
\label{sec:detailsofexample}

\begin{example}[The case with $I_{==} = \emptyset$, \nameref{as:constraint_qualification} holds]
    \label{eg:scs_case}
    Suppose $|I| = k$.
    Assume $I_{==} = \emptyset$, $I_{=>} = \{1,\dots, k\}$. 
    Let $\tcH = \cH_{\Ic\Ic}$, a square matrix of size $(n-k)$ and $\tG = G_\Ic$.
    Then 
    \begin{align}
        \cJ_\FPPE = [0_{k \times 1};  - \tcH\inv \tG] 
    \end{align}
    which is the same as \cref{eq:fppe_with_scs}. This agrees with the result from \citet{liao2023stat}. 
\end{example}

\begin{example}[The case with $ |I_{==}| = 2$.] \label{eg:case_two_buyer}
    Let $I_{=>} = \emptyset$, $I_{==} = \{ 1, 2\}$ and $\Ic = \{3, \dots, n\}$.
    Then 
    \begin{equation} \label{eq:def_two_buyer_J}
        \cJ_\FPPE \defeq 
        \begin{cases}
            DZ = - \cH \inv G & \text{if $Z_1 < 0, Z_2 < 0$} 
            \\
            D \begin{bmatrix}
                0 \\
                Z_{ 2}-\rho_{12} Z_{ 1} \\
                \vdots \\
                Z_{ n}-\rho_{1 n} Z_{ 1}
            \end{bmatrix}  & \text{if $Z_1 \geq 0, Z_2 - \rho_{12}Z_1 < 0$}
            \\
            D \begin{bmatrix}
                Z_{ 1}-\rho_{21} Z_{ 2} \\
                0 \\
                Z_{ 3}-\rho_{23} Z_{ 2}  \\
                \vdots \\
                Z_{ n}-\rho_{2 n} Z_{ 2}
            \end{bmatrix}  & \text{if $Z_2 \geq 0, Z_1 - \rho_{21}Z_2 < 0$} 
            \\
            [0_{2\times 1}; - \tcH \inv \tG]& \text{o.w.}
        \end{cases}
       ,
    \end{equation}
    where $\tcH = \cH_{\Ic\Ic}$ and $\tG = G_\Ic$. 
    We present the derivation in \cref{sec:detailsofexample}.
\end{example}

\subsubsection*{Deriving Closed-Form Expression for $\cJ_\FPPE$}

We recall a few definitions regarding the constraints.
Let $I = \{ i: \betasti = 1\}$, $\Ic = [n] \setminus I$. We further partition $I$ into 
\begin{align*}
    I_{=>} = \{i: \betasti =1, \deltasti > 0\}, I_{==} = \{i : \betasti = 1, \deltasti= 0\} .
\end{align*}
Let $A$ (resp.\ $B$) be a matrix whose rows are $e_i\tp, i\in I_{=>}$ (resp.\ $i \in I_{==}$). So $A$ is a $|I_{=>}| \times n$ matrix and $B$ is $|I_{==}| \times n$. 


A combinatorial expression of $\cJ_\FPPE$ is available.
One can solve the quadratic program \cref{eq:fppe_asym_program}, which contains linear inequality constraints, by solving 
at most $2^{|I_{==}|}$ linearly constrained programs. 
First, one create a candidate linearly constrained program by turning some inequality constraints to be equality ones, and then record the 
optimal objective value. Then the smallest value out of all $2^{|I_{==}|}$ candidate programs must 
be the same as the original program.

Given $G$, let
$Q_j $ and $ h_j $ be the optimal value and the optimal solution to the program 
\begin{align}
    \min_{h \in \Rn}  (h + \cH\inv G) \tp \cH (h + \cH\inv G) 
    \label{eq:criticalprogram}
    \text{ s.t. } [A;B_j]h = 0.
\end{align}
Here $B_j$ consists of some (possibly zero) rows of $B$, $j = 1, \dots, 2^{|I_{==}|}$.

The    program \cref{eq:criticalprogram} is just projecting the vector $-\cH\inv G$ onto the linear subspace spanned by $\Gamma_j = [A;B_j]$ w.r.t.\ the norm $\|\cd\|_\cH$. With this geometric interpretation, it is easy to write down the solution.
Define the projection matrix $  P_j                                                                                               = I - \cH \inv \Gamma_j\tp (\Gamma_j \cH \inv \Gamma_j\tp)\inv \Gamma_j$.
Then the closed-form expressions for $Q_j$ and $h_j$ are
\begin{align*}
     Q_j & = \| (I-P_j) \cH\inv G\|_{\cH}\sq = (\cH  \inv G)\tp \Gamma_j\tp (\Gamma_j \cH \inv \Gamma_j \tp )\inv \Gamma_j (\cH \inv G)
    \\
    h_j & = - P_j\cH \inv G,
\end{align*}
Then it is obvious that 
\begin{align}
    \cJ_\FPPE \defeq h_{j(G)} = - P_{j(G)}\cH \inv G
\end{align}
where $j(G) = \argmin _j \{Q_j : B h_j \leq 0 \}$.
Equivalently,
\begin{align}\label{eq:def_gen_J}
    & \cJ_\FPPE = 
    \\
    & \sum_{j=1}^{2^{|I_{==}|}}  -  \bigg(\indi(B h_j \leq 0) \prod_{\ell = 1}^{2^{|I_{==}|}} \indi ( Q_j \leq Q_\ell \text{ or } B h_\ell \not\leq 0)\bigg)P_j\cH \inv G
\end{align}
a random vector of length $n$. Only one of the term will be selected for each realization of $G$.
The representation allows us to derive the exact distribution in some cases.

\begin{proof}[Omitted details in \cref{eg:case_two_buyer}]
    We show $\cJ$ in \cref{eq:def_gen_J} reduces to the claimed expression \cref{eq:def_two_buyer_J}.

    Consider the programs
    \begin{align} 
        & \min_{h \in \Rn}  (h + \cH\inv G) \tp \cH (h + \cH\inv G)  
        \notag
            \\
        & \quad \text{subject to no constraints} \tag{$Q_1$}
            \\
        & \text{or subject to $h \tp e_1 = 0$} \tag{$Q_2$}
            \\
        & \text{or subject to $h \tp e_2 = 0$} \tag{$Q_3$}
            \\
        & \text{or subject to $h \tp e_1 = 0$, $h \tp e_2 = 0$} \tag{$Q_4$}
        \end{align} 
    
    For $Q_1$ the optimal solution is $h_1 = -\cH\inv G$.

    For $Q_2$, the optimal value is $ Q_2 = (\cH  \inv G)\tp e_1 (e_1\tp \cH \inv e_1 )\inv e_1\tp (\cH \inv G) = (G\tp \cH\inv e_1)\sq / (\cH\inv)_{11}\sq = Z_1\sq$ and the optimal solution $h_2$ is
    \begin{align*}
        h_2 
        & = - [I - \cH\inv e_1 (e_1\tp \cH\inv e_1) \inv e_t\tp] \cH\inv G
        \\
        & = \begin{bmatrix}
            0 & & & 
            \\
            \frac{(\cH\inv)_{21} }{ (\cH\inv)_{11}} & -1 & & 
            \\
            \vdots &  & \ddots & 
            \\
            \frac{(\cH\inv)_{n1} }{ (\cH\inv)_{11}} &  & & -1 
        \end{bmatrix} \cH\inv G
        \\
        & = D\begin{bmatrix}
            0 \\
             Z_{ 2} - \rho_{12} Z_{ 1} \\
            \vdots \\
            Z_{ n} -  \rho_{1 n} Z_{ 1}
        \end{bmatrix} 
    \end{align*}
    where we recall $Z = - D\inv \cH\inv G = [Z_1,\dots, Z_n]\tp $.
    For $Q_3$, the optimal value is $Z_2\sq$ and the optimal solution $h_3$ is the third display in \cref{eq:def_two_buyer_J}.

    For $Q_4$, the optimal solution $h_4$ is the fourth display in \cref{eq:def_two_buyer_J}.

    We consider the indicator part for $j = 2$ in \cref{eq:def_gen_J}, i.e., the expression
    \begin{align*}
        \indi(B h_2 \leq 0) \prod_{\ell = 1}^{4} \indi ( Q_2 \leq Q_\ell \text{ or } B h_\ell \not\leq 0)
    \end{align*}
    Then $B    = [e_1\tp; e_2\tp]$.
    Note $Bh_1 = [ Z_1 / \sqrt{(\cH\inv)_{11}}, Z_2 / \sqrt{(\cH\inv)_{22}} ]\tp$, $Bh_2 = [0,Z_2 - \rho_{12} Z_1]\tp$, $Bh_3 = [Z_1 - \rho_{12} Z_2, 0]\tp$.  Obviously, both $Q_2, Q_3 \geq Q_1$, and both $Q_2, Q_3 \leq Q_4$.
    It can be shown 
    \begin{align*}
        & \indi(B h_2 \leq 0) \indi (Q_1 \geq Q_2 \text{ or } B h_1 \not \leq 0)
        \indi (Q_3 \geq Q_2 \text{ or } B h_3\not \leq 0)
        \indi (Q_4 \geq Q_2 \text{ or } B h_4 \not \leq 0)
        \\
        & = 
        \indi(B h_2 \leq 0) \indi (B h_1 \not \leq  0)
        \indi (Q_3 \geq Q_2 \text{ or } B h_3\not \leq 0)
        \times 1
        \\
       & = \indi(Z_2 - \rho_{12} Z_1 \leq 0) \indi (Z_1 \geq  0) \indi (Z_2\sq \geq Z_1\sq \text{ or } Z_1 - \rho_{12} Z_2 > 0) \text{\quad almost surely}
        \\
       & = \indi ( Z_2 - \rho_{12} Z_1 \leq 0) \indi (Z_1\geq 0) \text{\quad almost surely}
        ,
    \end{align*}
    where the last equality follows by a case-by-case analysis.
    The indicator parts for $j = 1, 3, 4$ are analyzed similarly.
\end{proof}

\section{Experiments}
\label{sec:exp}

\subsection{Simulation: Verify Bootstrap Consistency for FPPE}
\label{sec:exp_sim}

In this section we verify the consistency of our bootstrap estimators, and investigate the effect of the bootstrap stepsize $\dtt$ (in \cref{eq:adboot_fppe}) on the quality of bootstrap approximation in FPPE on fully synthetic data.

We consider an $8$-buyer FPPE instance with 100 items sampled with i.i.d.\ values. Budgets of buyers are selected so that the first three buyers are unpaced ($\b=1$). This is to model the fact that in reality there could be buyers with leftover budgets.
We use dual averaging \citep{xiao2010dual,gao2020first,liao2022dualaveraging} to compute the limit FPPE \pmr  $\betast$. 
Finite FPPEs are computed with MOSEK.
We draw 100 finite FPPEs and obtain the finite FPPE distribution by plotting the histogram of $\sqrt{t} (\betagam - \betast)$. We call this \emph{true resampling}, which would not be possible in practice.
Finally, we then generate \emph{a single} FPPE and resample 100 bootstrapped $\beta$'s according to \cref{eq:adboot_fppe}, obtaining the bootstrap distribution estimate.
To experiment with different tail behaviors for values, we run three sets of experiments: uniform, exponential and truncated normal values. 
We also vary the choice of bootstrap stepsize $\epsilon_t = t^{-d}$. 

\textit{Results} In \cref{fig:sim1,fig:sim2,fig:sim3} we present the finite-sample distribution of $\betagam$ and $\betab$. 
Each column corresponds to the \pmr of a buyer,
and each row corresponds to a choice of $d$.
First, we observe that with a suitable choice of $d$, the bootstrap distribution is a good approximation to that of finite FPPE with true resampling.
For buyers with $\betasti = 1$ the proposed bootstrap is able to correctly identify them. 
For buyers with $\betasti < 1$, bootstrap correctly captures the range and the shape of the distribution.
This result is significant, because to obtain the distributions of FPPE, we need to observe multiple market equilibria, to which we usually do not have access.
The bootstrap distribution, on the other hand, is generated based on just one finite FPPE.
Second, we also observe that if $\epsilon_t$ is too large ($d$ too small), the quality of approximation degrades. In particular, in the bottom rows of plots for uniform and normal values, bootstrap tends to ignore the right part of the distribution of finite FPPEs.

\begin{figure}[ht]
    \centering
    \includegraphics[scale= 0.5]{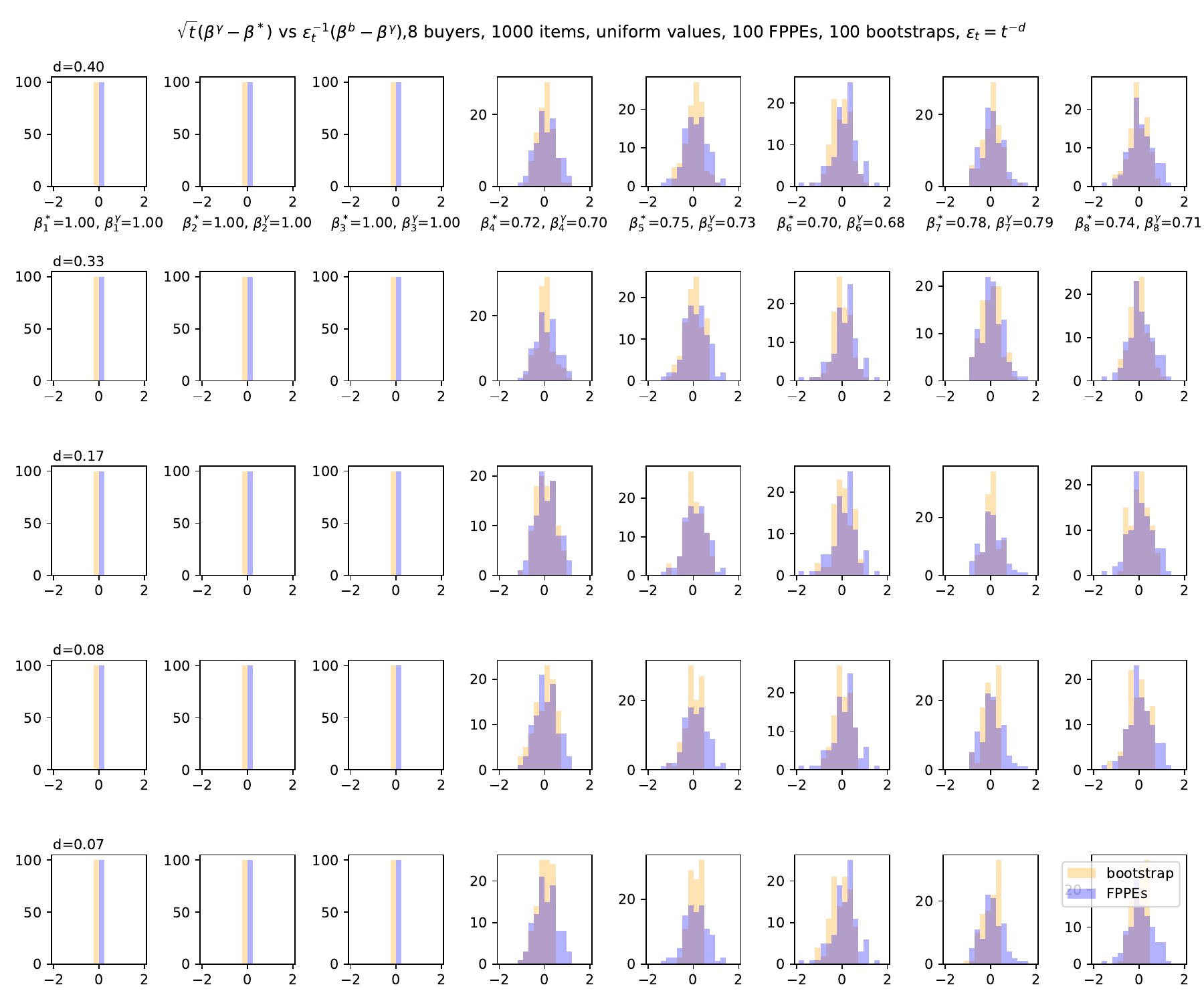}
    \caption{Comparison of Bootstrap and FPPE finite item distribution.}
    \label{fig:sim1}
\end{figure}
\begin{figure}[ht]
    \centering
    \includegraphics[scale = 0.5]{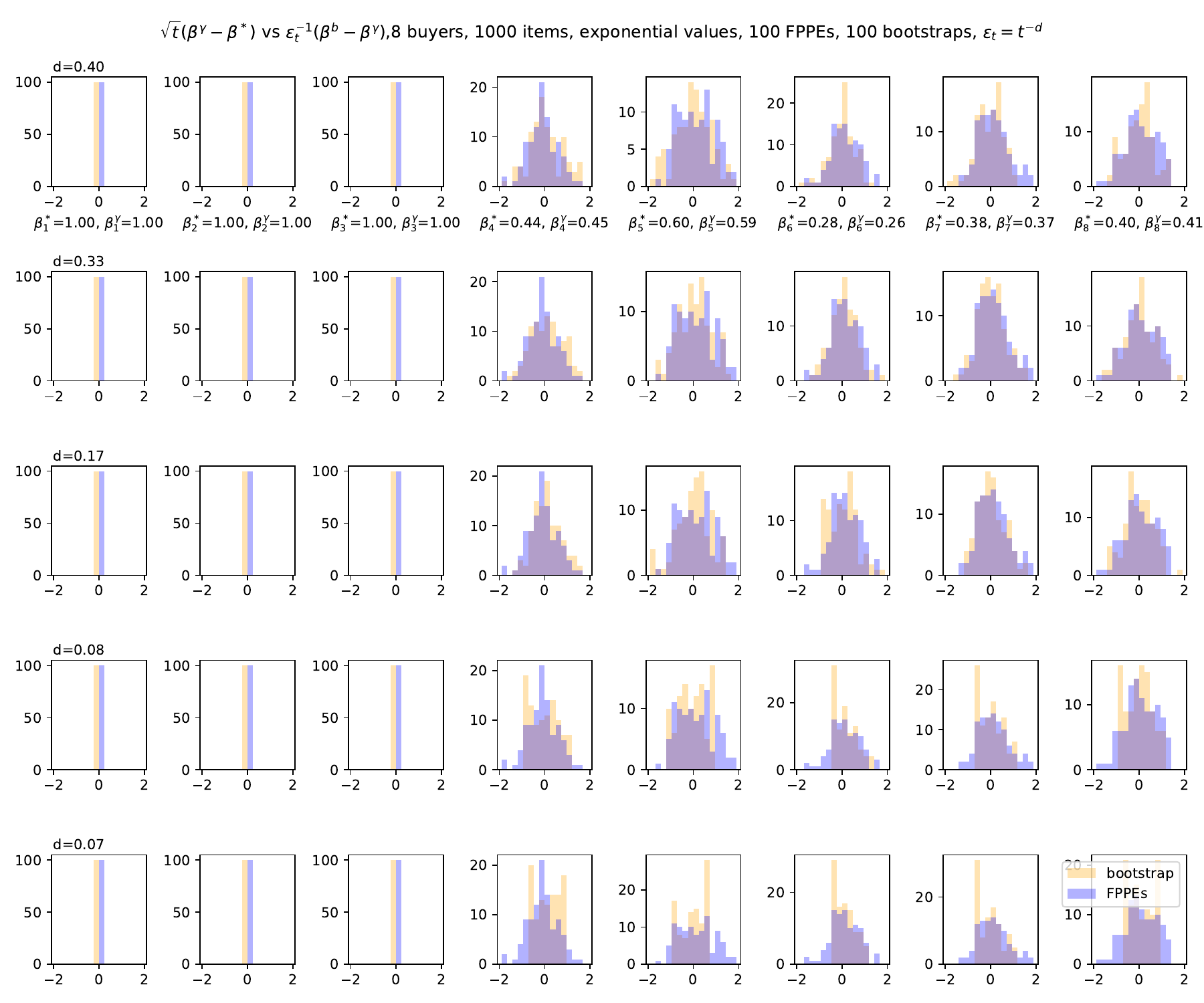}
    \caption{Comparison of Bootstrap and FPPE finite item distribution.}
    \label{fig:sim2}
\end{figure}
\begin{figure}[ht]
    \centering
    \includegraphics[scale = 0.5]{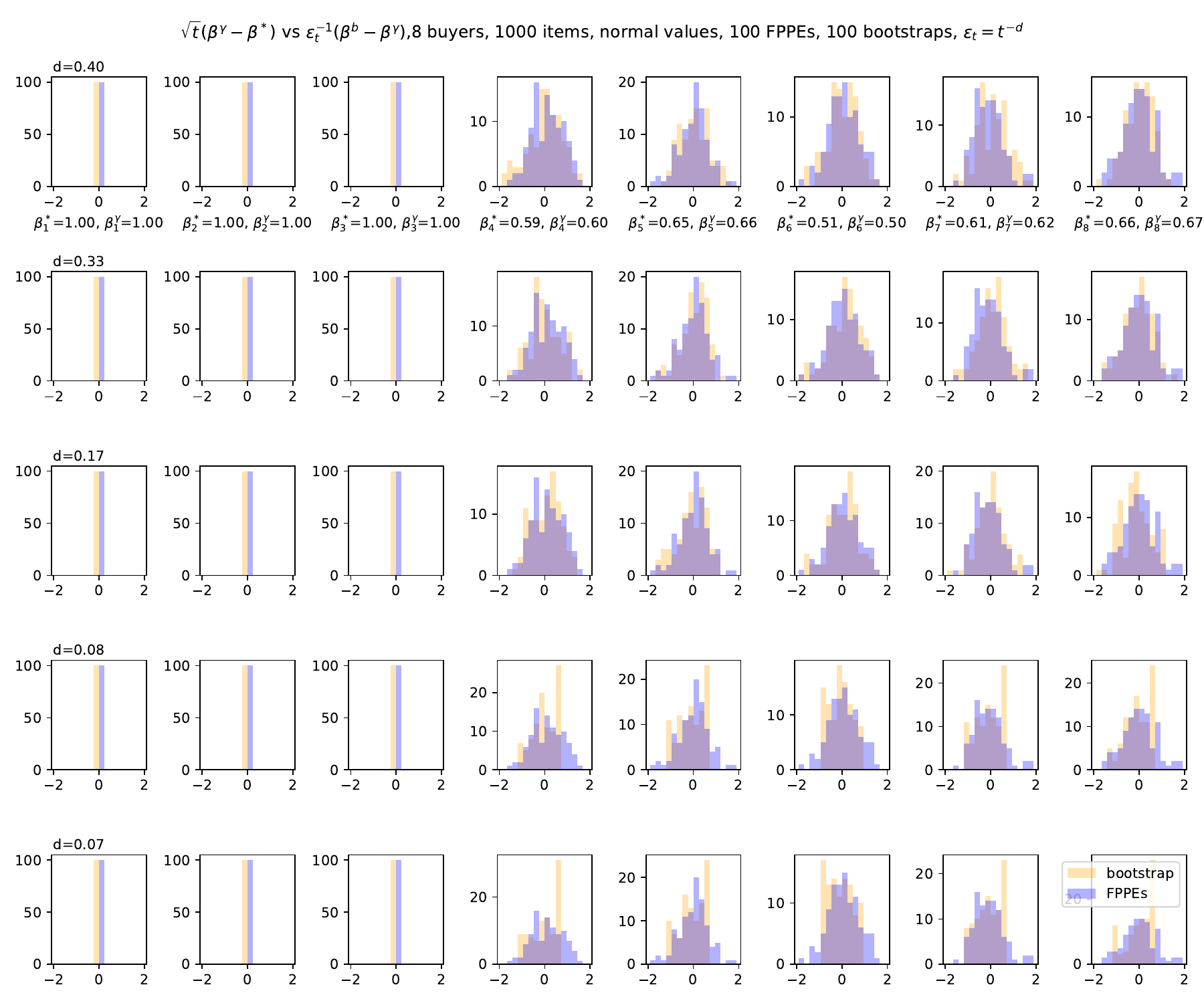}
    \caption{Comparison of Bootstrap and FPPE finite item distribution.}
    \label{fig:sim3}
\end{figure}

\subsection{Semi-Real Data}
\label{sec:exp_ipinyou}
In this section we apply our bootstrap estimator to a real-world dataset, the iPinYou dataset \citep{liao2014ipinyou}.

\textit{The data.} The iPinYou dataset \citep{liao2014ipinyou} contains raw log data of the bid, impression, click, and conversion history on the iPinYou platform
in the weeks of 
March 11--17, 
June 8--15
and October 19--27.
We use the impression and click data of 5 advertisers on June 6, 2013, containing a total of 1.8 million impressions and 1,200 clicks.
As in the main text, let $i \in \{1,2,3,4,5\}$ index advertisers (buyers) and let $\tau$ index impressions/users (items in FPPE terminology).
The five advertiser are labeled by number and their categories are revealed: 1459 (Chinese e-commerce), 3358 (software), 3386 (international e-commerce) and 3476 (tire).
From the raw log data, the following dataset can be extracted.
The response variable is a binary variable $ \click_i^\tau \in \{0,1\}$ that indicates whether the user clicked the ad or not. 
The relevant predictors include 
a categorical variable $\AE$ of three levels that records from which ad-exchange the impression 
was generated, a categorical variable $\RG$ of 35 levels indicating provinces of user IPs, and finally 44 boolean variables, $\TG$'s, indicating whether a user belongs to certain user groups  defined based on 
demographic, geographic and other information.
We select the top-10 most frequent user tags and denote them by $\TG_1,\dots,\TG_{10} \in \{0,1\}$. 
Both $\AE$ and $\TG$ are masked in the dataset, and we do not know their real-world meanings.

\textit{Simulate advertisers with logistic regression.}
The raw data contains only five advertisers. In order to simulate more realistic advertiser values, we fit a logistic regression and then perturb the fitted coefficients to generate more advertisers.
We posit the following logistic regression model for click-through rates (CTRs).
For a user $\tau$ that saw the ad of advertiser $i$, the click process is governed by
\begin{align*}
   & \CTR_i^\tau = \P(\click_i^\tau = 1 \given \thetau) = 
    \frac{1}{1+\exp (w_i\tp \theta^\tau)} 
   \\
   &\theta^\tau = [1, \AE_2, \AE_3, \RG_2,\dots, \RG_{35}, \TG_1, \dots, \TG_{10}]\in \{1\} \times \{0,1\}^{46}
\end{align*}
where the weight vectors $w_i \in \R^{47}$ are the coefficients to be estimated from the data.
Note that $\AE_1$ and $\RG_1$ are absorbed in the intercept.
By running 5 logistic regressions, we obtain regression coefficients $w_1, w_2, \dots, w_5$.
To visualize the fitted regression, in \cref{fig:real} we show the estimated click-through rate distributions of the five advertisers. The diagonal plots are the histogram of CTRs, and the off-diagonal panels are the pair-wise scatter plots of CTRs.
To generate more advertisers, we take a convex combination of the coefficients $w_i$'s, add uniform noise, and obtain a new parameter, say $w'$. Given an item, the CTR of the newly generated advertisers will be $\frac{1}{1 + \exp(\theta \tp w')}$. 
The limit value distribution in \cref{def:limit_fppe} is the historical distribution of the simulated advertisers' predicted CTRs of the 1.8 million impressions.

\textit{Experiment setup.}
In this section we aim to produce confidence interval of the sum $\sumi \betasti$ with 
the bootstrap estimator \cref{eq:adboot_fppe}.
Firstly, the sum equals $n$ times the average price-per-utility of advertisers, a measure of efficiency of the system. Secondly, since most quantities in FPPE, such as revenue and social welfare, are smooth functions of \pmrs, being able to perform inference about a linear combination of $\beta$'s indicates the ability to infer first-order estimates of those quantities.

The estimator requires an initial consistent estimate of the Hessian matrix, which is implemented with finite difference in \cref{eq:def_Gb_H} with differencing stepsize $\eps = t ^{-0.4}$.
The estimator also requires a bootstrap stepsize $\eps_t = t^{-d}$. We try $d$ over the grid $\{ 0.4, 0.3, 0.2, 0.1, 0.05 \}$.

An experiment has parameters $(t,n, d, \alpha)$.
Here $t \in \{ 100, 300, 500 \}$ is the number of items and $n \in \{ 10, 20, 30, 50 \}$ the number of advertisers.
Parameter $d$ is the exponent of the bootstrap stepsize, and $\alpha \in \{0.1, 0.3, 0.5 \}$ is the proportion of advertisers that are not budget-constrained (i.e., $\beta = 1$).
To control $\alpha$ in the experiments, we select budgets as follows. Give infinite budgets to the first $\lfloor \alpha n \rfloor $ advertisers. Initialize the rest of the advertisers' budgets randomly, and keep decreasing their budgets until their \pmrs are strictly less than 1.
For the experiment $(t,n, d, \alpha)$, we first compute the \pmr in the limit market using dual averaging \citep{xiao2010dual,gao2021online,liao2022dualaveraging}.
In one simulation of the experiment $(t,n, d, \alpha)$,
we sample one FPPE by drawing values from the limit value distribution.
Now given one FPPE, we generate bootstrapped \pmrs $\{ \beta^{b,1}, \dots, \beta^{b,B}\}$ by \cref{eq:adboot_fppe}. We calculate the set of sums $S =\{ s^{b,1},\dots, s^{b,B}\}$ where $s^{b,1} = \sumi \beta^{b,1}_i$ and so on.
To obtain a confidence interval with nominal coverage $95\%$,
we let $\ell, u$ be the 2.5\% and 97.5\% percentiles of $S$. We report the coverage rate and the width of $[\ell, u]$ in \cref{tbl:bsipinyou}. We perform $B = 100$ bootstrap replications in each simulation. The reported coverage rate for an experiment with parameters $ (t,n, d, \alpha)$ is averaged over 100 simulations.

\textit{Results.}
For an appropriate choice of $d \in [0.2, 0.3]$, the finite-sample coverage rate agrees with the nominal coverage 95\%.
Although our theory suggests that as long as $d < 1/2$, the bootstrapped distribution is asymptotically consistent, parameter~$d$ does affect finite-sample coverage.
Too small a $d$ (for example, 0.10 or 0.05) results in over-coverage and a large $d$ results in under-coverage.
We also observe that for $d = 0.4$ and $n=50$, the finite-sample coverage rate
is undesirable for item size $t=100$. 
Reassuringly, it
increases to a nominal coverage of 95\% as item size increases.
We also see that the width of the confidence interval decreases as 
the number of items increases while maintaining nominal coverage. This is expected since the interval width decreases at a rate of $1/\sqrt{t}$.
Finally, for appropriately chosen $d$ and item size $t$, the proportion of unpaced advertisers $\alpha$ does not affect finite-sample coverage rates, which demonstrates the robustness of the proposed bootstrap estimator.
\begin{landscape}
    \input{tbl_bsipinyou_quantile.tex}
\end{landscape}

\begin{figure}
    \centering
    \includegraphics[scale = 0.6]{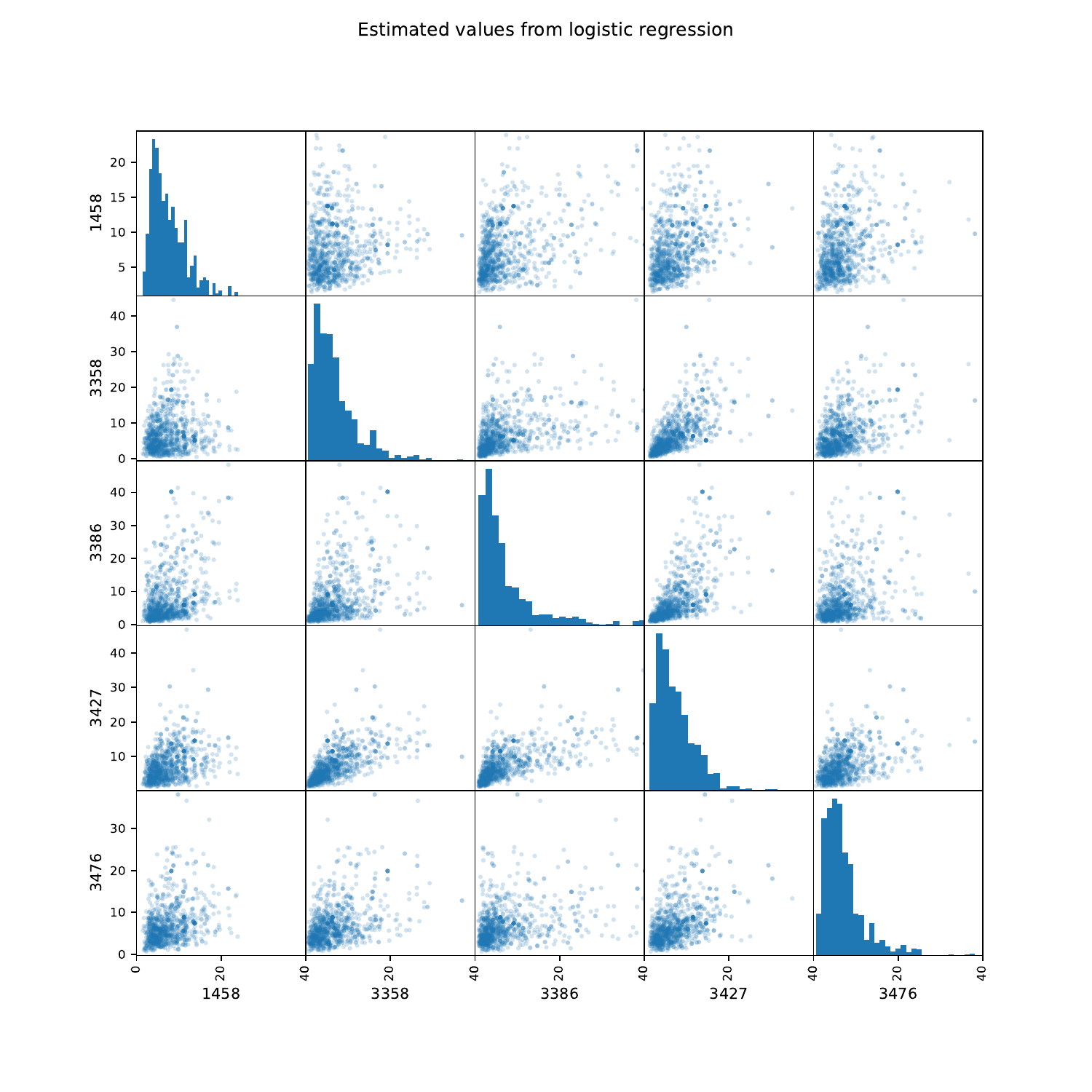}
    \caption{Click-through rate (in 0.01\%) distributions from logistic regression. }
    \label{fig:real}
\end{figure}

\section{Review of Weak Epi-Convergence}

The Hoffman-J\o rgensen weak convergence theory is a powerful tool to study the asymptotics convergence of statistical functional, especially the argmin functional.
To apply the theory to specific applications, one needs to instantiate it with a metric space, usually a space of functions, verify continuity or some form of differentiability of the statistical functional and the weak convergence of certain processes, and then finally invoke continuous mapping theorem or functional delta method. 
Common choices of metric spaces include the space of bounded functions (on some metric space) with the uniform metric \citep{van2000asymptotic,kosorok2008introduction,billingsley2013convergence}, 
the space of locally bounded functions on $\Rn$ with the topology of uniform convergence on compacta \citep{kim1990cube} or the topology of induced by the hypi-semimetric \citep{bucher2014uniform},
the space of positive signed measures or point measures on the real line with an appropriate metric \citep{resnick2008extreme},
the space of bounded real-valued continuous functions on $\Rn$ with the usual sup-norm \citep{gine2008uniform},
or the space of cadlag (left limit right continuous) functions on $\R$ with the Skorohod $J_1$ metric \citep{skorokhod1956limit,pollard2012convergence}.
To study asymptotics of constrained minimizers, a suitable choice of metric space is the space of 
extended real-valued lower semi-continuous functions with  the metric that induces the topology of epi-convergence. Such an approach dates back to \citet{geyer1994asymptotics} and \citet{molchanov2005theory}, and
used in \citet{chernozhukov2004likelihood} for nonregular models and  \citet{chernozhukov2007estimation} for set estimation,
and
more recently is used by \citet{parker2019asymptotic} for constrained quantile regression, and \citet{li2020numerical} and \citet{li2023proximal} in the context of bootstrap.

To begin with, we introduce the concept of epi-convergence and its probabilistic extensions.
Consider
\begin{align*}
    & \cL_n = \{ f:\Rn \to \bar \R: f \text{ is proper lower semi-continuous (lsc)} \} ,
    \\
    &  \cC\cS_n = \{ A : A\text{ is a nonempty closed set in }  \Rn\} .
\end{align*}
For $f:\Rn\to\bar \R$, we let $\epi f \defeq \{ (x,v) \in \R^{n+1} : f(x) \leq v \}$ be its epi-graph. Also for $C$ a nonempty closed subset of $\R^{n+1}$, and a point $v \in \R^{n+1}$, let $d_C(v) \defeq \inf\{\|u-v\|_2: u\in C \}$ be the distance of $v$ to the set $C$, 
and for nonempty sets $A$ and $B$, define $d_\rho (A, B) \defeq \max \left\{\left|d_{A}(v)-d_{B}(v)\right|:\|v \|_2 \leq \rho\right\}$.
Define the Attouch-Wets metric on $\cC\cS_{n+1}$ by
\begin{align}
    \label{eq:defdaw}
    d_{\text{AW}} (A, B ) \defeq  \int_0^{\infty} d_\rho(A , B)  \exp (-\rho) \diff \rho
\end{align}
And for $f,g \in \cL_n$, define the metric 
$d_\epi (f,g) \defeq d_\text{AW}(\epi f, \epi g)$.
In $\cC\cS_n$, the topology induced by $d_{\text{AW}}$ is equivalent to 
Wijsman topology and the topology of Painlev\'e-Kuratowski set convergence \citep{romisch2004delta}.
The metric space $(\cC\cS_n, d_\text{AW})$ is complete and separable \citep[Theorem 4.42, Proposition 4.45]{rockafellar2009variational}.
Also, the metric space $(\cL_n, d_\epi )$ is complete and separable \citep[Theorem 7.58]{rockafellar2009variational}. We say a sequence of functions $f_t \in \cL_n$ epi-converges to $f\in \cL_n$ if $d_\epi (f_t, f) \to 0$.

\begin{defn}[Epi-convergence in probability]

Let $Z_t: \Omega\times \Rn \to \bar \R$ and $Z:\Omega\times \Rn \to \bar \R$ be random lsc, extended real-valued functions. 
We write $Z_t \toepi Z$ in probability if for any $\epsilon > 0$ it holds $\P( \omega : d_\epi(Z_t(\omega, \cd),Z(\omega, \cd)) > \epsilon) \to 0$ as $t\to \infty$.

\end{defn}



\begin{defn}[Epi-convergence in distribution, \citet{knight1999epi}]
    We say $Z_t \todepi Z$ if for any closed rectangles $R_1, \dots, R_k$ with open interiors $R_1^o, \dots, R_k^o$, it holds 
    the random vector $(\inf_{R_j} Z_t(\cd), j = 1,\dots, k) \tod (\inf_{R_j} Z(\cd), j = 1,\dots, k) $ and 
    $(\inf_{R_j^o} Z_t(\cd), j = 1,\dots, k) \tod (\inf_{R_j^o} Z(\cd), j = 1,\dots, k) $, or equivalently, for any real numbers $a_1,\dots, a_k$,  
\begin{align*}
    \P & \left(\inf _{u \in R_1} Z(u)>a_1, \cdots, \inf _{u \in R_k} Z(u)>a_k\right) \\
    & \leq \liminf _{t \rightarrow \infty} \P\left(\inf _{u \in R_1} Z_t(u)>a_1, \cdots, \inf _{u \in R_k} Z_t(u)>a_k\right) \\
    & \leq \limsup _{t \rightarrow \infty} \P\left(\inf _{u \in R_1^o} Z_t(u) \geq a_1, \cdots, \inf _{u \in R_k^o} Z_t(u) \geq a_k\right) \\
    & \leq \P\left(\inf _{u \in R_1^o} Z(u) \geq a_1, \cdots, \inf _{u \in R_k^o} Z(u) \geq a_k\right)
\end{align*} 

By Corollary 2.4 from \citet{pflug1991asymptotic}, $Z_t \todepi Z$ is equivalent to the weak convergence of $\epi Z_t$ to $\epi Z$ in the metric space $(\cC\cS_{n+1}, d_{AW})$.

\end{defn}

The following lemma is Slutsky's theorem \citep[Theorem 7.15]{kosorok2008introduction} specialized to the space $(\cL_n, d_\epi)$.
\begin{lemma}
    \label{thm:slutskyepi}
    Let $(Z_t)_t, (Y_t)_t$ and $Z$ be
    random lsc, extended real-valued functions 
    and $Y$ be a deterministic element in $\cL_n$.
    If $Y_t \toepi Y$ in probability and $Z_t \todepi Z$. Then $Z_t + Y_t \todepi Z + Y$.
\end{lemma}

We also introduce a bootstrap version of weak epi-convergence following Section 2.2.3 in \citet{kosorok2008introduction}.

\begin{defn}[Conditional weak epi-convergence in probability]
    Let $BL$ denote the space of Lipschitz functions $f: (\cL_n, d_{\text{AW}}) \to \R$ with Lipschitz parameter equal 1, i.e., $\sup |f| \leq 1$ and $|f(x) - f(y)| \leq d_{AW}(x,y)$. 
    And suppose $Z_t = Z_t(X, W)$ is defined on a product probability space, where $W$ represents bootstrap weights, and $X$ represents data.
    The process $Z_t$ converges to $Z$ in the sense of weak  epi-convergence conditionally in probability if  
    $\sup_{f \in BL} |\E_W[f(Z_t)| X ] - \E[f(Z)]| \to 0$ in probability, along with certain measurability conditions.

\end{defn}

A bootstrap version of argmin continuous mapping theorem for weak epi-convergence can also be stated \citep{li2020numerical}; see Lemma 10.1 and the proof of Theorem 4.2 in \citep{li2020numerical}.
We will also need a bootstrap version of Slutsky's theorem for weak epi-convergence \citet[Theorem 4]{knight1999epi}.

%% file: tbl_bsipinyou_quantile.tex
\begin{table}[ht]
\centering
\caption{Coverage and width of CI in parentheses. $\alpha$ = proportion of $\beta_i = 1$, $d$ is the exponent in bootstrap stepsize $\epsilon_t = t^{-d}$.}
\label{tbl:bsipinyou}
\begin{tabular}{llllllllllllll}
\toprule
     & $\alpha$ & \multicolumn{4}{l}{0.1} & \multicolumn{4}{l}{0.3} & \multicolumn{4}{l}{0.5} \\
     & buyer &                       10 &                       20 &                       30 &                        50 &                       10 &                       20 &                       30 &                        50 &                       10 &                       20 &                       30 &                       50 \\
\midrule $d$ & item &                          &                          &                          &                           &                          &                          &                          &                           &                          &                          &                          &                          \\
\midrule
\multirow{3}{*}{0.40} & 100 &  \makecell{0.79\\(1.07)} &  \makecell{0.82\\(1.94)} &  \makecell{0.81\\(3.02)} &   \makecell{0.88\\(5.17)} &  \makecell{0.82\\(1.20)} &  \makecell{0.84\\(2.02)} &  \makecell{0.82\\(3.47)} &    \makecell{0.3\\(4.53)} &   \makecell{0.7\\(1.11)} &  \makecell{0.91\\(2.18)} &  \makecell{0.86\\(2.67)} &  \makecell{0.36\\(3.44)} \\
     & 300 &  \makecell{0.73\\(0.67)} &  \makecell{0.82\\(1.14)} &  \makecell{0.85\\(1.67)} &   \makecell{0.85\\(2.74)} &  \makecell{0.89\\(0.83)} &  \makecell{0.76\\(1.21)} &   \makecell{0.9\\(2.24)} &   \makecell{0.53\\(3.17)} &  \makecell{0.82\\(0.74)} &  \makecell{0.85\\(1.51)} &  \makecell{0.95\\(2.04)} &  \makecell{0.49\\(2.65)} \\
     & 500 &  \makecell{0.78\\(0.55)} &  \makecell{0.83\\(0.92)} &  \makecell{0.79\\(1.34)} &   \makecell{0.89\\(2.13)} &  \makecell{0.88\\(0.68)} &  \makecell{0.83\\(0.97)} &  \makecell{0.88\\(1.87)} &   \makecell{0.83\\(2.78)} &   \makecell{0.8\\(0.60)} &  \makecell{0.85\\(1.28)} &  \makecell{0.91\\(1.69)} &  \makecell{0.81\\(2.39)} \\
\cline{1-14}
\multirow{3}{*}{0.30} & 100 &  \makecell{0.94\\(1.68)} &  \makecell{0.94\\(3.05)} &  \makecell{0.98\\(4.53)} &   \makecell{0.97\\(7.75)} &  \makecell{0.94\\(1.81)} &  \makecell{0.93\\(3.01)} &  \makecell{0.97\\(4.95)} &    \makecell{0.8\\(6.20)} &  \makecell{0.85\\(1.67)} &  \makecell{0.95\\(3.16)} &  \makecell{0.98\\(4.00)} &  \makecell{0.77\\(4.89)} \\
     & 300 &  \makecell{0.96\\(1.21)} &  \makecell{0.96\\(2.01)} &  \makecell{0.97\\(2.89)} &   \makecell{0.97\\(4.69)} &  \makecell{0.96\\(1.40)} &  \makecell{0.93\\(2.08)} &  \makecell{0.99\\(3.76)} &   \makecell{0.99\\(5.08)} &  \makecell{0.91\\(1.28)} &  \makecell{0.91\\(2.39)} &  \makecell{0.99\\(3.27)} &  \makecell{0.93\\(4.32)} \\
     & 500 &  \makecell{0.95\\(1.03)} &  \makecell{0.97\\(1.71)} &  \makecell{0.98\\(2.40)} &    \makecell{1.0\\(3.91)} &  \makecell{0.98\\(1.24)} &  \makecell{0.98\\(1.73)} &  \makecell{0.99\\(3.24)} &    \makecell{1.0\\(4.69)} &  \makecell{0.93\\(1.10)} &  \makecell{0.96\\(2.20)} &  \makecell{0.98\\(2.90)} &  \makecell{0.99\\(4.09)} \\
\cline{1-14}
\multirow{3}{*}{0.20} & 100 &  \makecell{0.99\\(2.39)} &  \makecell{0.99\\(4.48)} &   \makecell{1.0\\(6.45)} &  \makecell{0.98\\(10.86)} &  \makecell{0.99\\(2.63)} &   \makecell{1.0\\(4.41)} &  \makecell{0.99\\(6.51)} &   \makecell{0.96\\(8.26)} &  \makecell{0.96\\(2.41)} &   \makecell{1.0\\(4.39)} &   \makecell{1.0\\(4.94)} &  \makecell{0.96\\(6.43)} \\
     & 300 &   \makecell{1.0\\(1.97)} &  \makecell{0.99\\(3.33)} &  \makecell{0.99\\(4.77)} &    \makecell{1.0\\(7.63)} &  \makecell{0.99\\(2.19)} &   \makecell{1.0\\(3.40)} &   \makecell{1.0\\(5.54)} &    \makecell{1.0\\(7.36)} &  \makecell{0.96\\(2.11)} &  \makecell{0.98\\(3.65)} &  \makecell{0.99\\(4.70)} &   \makecell{1.0\\(6.01)} \\
     & 500 &   \makecell{1.0\\(1.80)} &  \makecell{0.99\\(3.05)} &  \makecell{0.99\\(4.13)} &    \makecell{1.0\\(6.57)} &   \makecell{1.0\\(2.09)} &  \makecell{0.99\\(3.03)} &   \makecell{1.0\\(5.14)} &    \makecell{1.0\\(7.04)} &  \makecell{0.99\\(1.90)} &  \makecell{0.95\\(3.57)} &  \makecell{0.99\\(4.55)} &   \makecell{1.0\\(5.95)} \\
\cline{1-14}
\multirow{3}{*}{0.10} & 100 &   \makecell{1.0\\(3.36)} &   \makecell{1.0\\(5.99)} &   \makecell{1.0\\(8.56)} &   \makecell{1.0\\(13.93)} &   \makecell{1.0\\(3.44)} &   \makecell{1.0\\(5.97)} &   \makecell{1.0\\(8.04)} &   \makecell{1.0\\(10.07)} &  \makecell{0.97\\(3.24)} &  \makecell{0.99\\(5.30)} &  \makecell{0.99\\(6.18)} &  \makecell{0.98\\(7.52)} \\
     & 300 &   \makecell{1.0\\(2.98)} &   \makecell{1.0\\(5.01)} &   \makecell{1.0\\(6.89)} &   \makecell{1.0\\(10.81)} &   \makecell{1.0\\(3.20)} &   \makecell{1.0\\(4.83)} &   \makecell{1.0\\(7.29)} &    \makecell{1.0\\(9.09)} &   \makecell{1.0\\(3.02)} &  \makecell{0.94\\(4.79)} &  \makecell{0.92\\(6.04)} &   \makecell{1.0\\(7.67)} \\
     & 500 &   \makecell{1.0\\(2.87)} &   \makecell{1.0\\(4.50)} &   \makecell{1.0\\(6.41)} &    \makecell{1.0\\(9.86)} &   \makecell{1.0\\(3.09)} &   \makecell{1.0\\(4.60)} &   \makecell{1.0\\(6.94)} &    \makecell{1.0\\(9.39)} &  \makecell{0.99\\(2.85)} &  \makecell{0.98\\(4.80)} &  \makecell{0.93\\(6.02)} &  \makecell{0.94\\(7.52)} \\
\cline{1-14}
\multirow{3}{*}{0.05} & 100 &   \makecell{1.0\\(3.70)} &   \makecell{1.0\\(6.73)} &   \makecell{1.0\\(9.61)} &   \makecell{1.0\\(15.67)} &   \makecell{1.0\\(3.82)} &  \makecell{0.99\\(6.59)} &   \makecell{1.0\\(8.46)} &   \makecell{1.0\\(10.66)} &  \makecell{0.98\\(3.67)} &  \makecell{0.95\\(5.59)} &  \makecell{0.93\\(6.47)} &   \makecell{1.0\\(8.10)} \\
     & 300 &   \makecell{1.0\\(3.47)} &   \makecell{1.0\\(5.86)} &   \makecell{1.0\\(8.06)} &   \makecell{1.0\\(12.21)} &   \makecell{1.0\\(3.57)} &   \makecell{1.0\\(5.72)} &   \makecell{1.0\\(8.06)} &  \makecell{0.99\\(10.36)} &   \makecell{1.0\\(3.34)} &  \makecell{0.97\\(5.48)} &   \makecell{0.9\\(6.45)} &  \makecell{0.94\\(8.22)} \\
     & 500 &   \makecell{1.0\\(3.35)} &   \makecell{1.0\\(5.45)} &   \makecell{1.0\\(7.38)} &   \makecell{1.0\\(11.32)} &   \makecell{1.0\\(3.61)} &   \makecell{1.0\\(5.43)} &   \makecell{1.0\\(7.74)} &    \makecell{1.0\\(9.93)} &  \makecell{0.99\\(3.33)} &  \makecell{0.97\\(5.24)} &  \makecell{0.93\\(6.50)} &  \makecell{0.98\\(8.13)} \\
\bottomrule
\end{tabular}
\end{table}|

%% file: proofs.tex
\section{Proofs}

\subsection{Stochastic Equicontinuity Results for the EG Objective}

Let $\ept = o(1)$ and $K$ be a compact set.
Let $D_F(\t,\b) \in \partial F(\t,\b)$ be a deterministic element of the subgradient.
Note by \nameref{as:twice_diff}
$D^*_F(\cd) = D_F(\cd,\betast) = \nabla F(\cd, \betast)$.
We also let $F(\b) = F(\cd, \b)$. 
Note that in the following claims, we do not need $\nabla H(\betast) = 0$. They work for any $\betast$ at which $H$ is continuously differentiable in a neighborhood.
\begin{claim}
    \label{lm:stocdiff_boot}
    \begin{align*}
        \sup_{h\in K} (P_t - P) ( F(\betast + \ept h) - F(\betast) - \ept h \tp D_F\st(\cd) ) = o_p(\ept / \sqrt t)
        \\
        \sup_{h\in K} (\Pexbt - P_t) ( F(\betast + \ept h) - F(\betast) - \ept h \tp D_F\st(\cd) ) = o_p(\ept / \sqrt t)
    \end{align*}
\end{claim}

\begin{proof}[Proof of \cref{lm:stocdiff_boot}]
        Let $ r_{1,F}(\cdot, \beta) = F(\cdot, \beta) - F(\cd,\betast) - D^*_F(\cdot) \tp (\beta - \betast)$.

        By \nameref{as:twice_diff} there is a neighborhood of $\betast$, say $N$, on which $H$ is differentiable. Then for any $\beta\in N$, the set $\{\theta:\beta\mapsto f(\theta, \beta) \text{ differentiable at } \beta\}$ is measure one.
        Choose $t$ large enough so that the ball $\{\beta:\|\beta-\betast\nmt \leq \delta_t \}$ is contained in $N$.
        By a mean value theorem for locally Lipschitz functions, \citep[Theorem 2.3.7]{clarke1990optimization},
        it holds $(P_t-P) (F(\beta) - F(\betast) ) = \zeta \tp (\beta-\betast)$ where $\zeta \in\partial(P_t - P) F(\tilde \beta)$ and $\tilde \beta$ lies on the segment joining $\beta$ and $\betast$.
        By $\tilde \beta \in N$,                    it holds $\zeta = (P_t - P) D_F(\cd, \tilde \beta)$.
        Then the desired claim is equivalent to
        \begin{align}
            & \sup_{ \|\beta-\betast\|_2\leq \delta_t}  \frac{(\Pt - P) r_{1,F}(\cd, \beta)}{ \frac{1}{\sqrt t} \|\beta-\betast\|_2  } 
            \notag
            \\
            & = 
            \sup_{ \|\beta-\betast\|_2\leq \delta_t}  
            \frac{(\Pt - P) (D_F(\cd, \tilde \beta) - D_F(\cd,\betast))\tp (\tilde \beta - \betast) }{\frac{1}{\sqrt t}\|\beta-\betast\|_2  }  
            \notag
            \\
            & \leq \sup_{ \|\beta-\betast\|_2\leq \delta_t}  \| \sqrt t (P_t - P)(D_F(\cd, \tilde \beta) - D_F(\cd,\betast)) \nmt
             =  o_p(1)
            \label{eq:dfstocheq}
        \end{align}
    where the last equality is due to \citet{liao2023stat}.

    The assumption on the bootstrap weights (\cref{as:exboot}) implies that a bootstrap version of \cref{eq:dfstocheq} holds, i.e.,  
    $\sup_{ \|\beta-\betast\|_2\leq \delta_t}  \| \sqrt t (\Pexbt - \Pt)(D_F(\cd, \tilde \beta) - D_F(\cd,\betast)) \nmt
     =  o_p(1)$ \citep[Lemma 4.1]{wellner1996bootstrapping}.
    The same argument goes through for the proof of bootstrap differentiability. We finish the proof of the lemma.
\end{proof}
\begin{claim}
    \label{lm:stochdiff_of_F}
    \begin{align*}
        \sup_{\|h - s \nmt = o(1), s, h \in K} (P_t - P) (F(\betast + \ept h ) - F(\betast  + \ept s))
        = o_p(\ept / \sqrt t)
        \\
        \sup_{\|h - s \nmt = o(1), s, h \in K} (\Pexbt - P_t) (F(\betast + \ept h ) - F(\betast  + \ept s))
        = o_p(\ept / \sqrt t)
    \end{align*} 
\end{claim}
\begin{proof}
    This is implied by \cref{lm:stocdiff_boot}.
\end{proof}
\begin{claim} 
    \label{claim:htquadraticexpansion}
    \begin{align*}
        \sup_{h \in K} P_t ( F(\betast + \ept h) - F(\betast) - \ept h \tp D\st(\cd) - \ept \sq h \tp \cH h ) = o_p(\ept / \sqrt t + \ept \sq)
    \end{align*}
\end{claim}
\begin{proof}
    Let $r_{2,F}(h)\defeq F(\cdot, \betast + h) - F(\cdot, \betast) - D_F(\cdot,\betast)\tp h - \frac12 h\tp \cH h$.

    Split the LHS by $\sup_h | P_t r_{2,F} (\ept h) | \leq \sup_h |P r_{2,F} (\ept h)| + \sup_h |(P_t - P)r_{2,F}(\ept h)| $. The first term is $o(\ept \sq)$ by twice differentiability.
    The second term is bounded by $o_p(\ept / \sqrt t)$ as in \cref{lm:stocdiff_boot}.
\end{proof}

\textbf{Notations in the proof sections}.
Define the demeaned and the centered function 
$\tilde F(\cdot, \beta) = F(\cdot,\beta) - \E[F(\theta, \beta)]$ and
$\bar   F(\cd, \beta)   = F(\cd, \beta) - F(\cd, \betast)$. Let $H^b (\beta) = \Pbt F(\cd, \beta)$ be the bootstrapped EG objective.
\subsection{Proof of \cref{thm:exboot_fm}}
\label{sec:proof:thm:exboot_fm}
\begin{proof}[Proof of \cref{thm:exboot_fm}]
    In the proof, we use $\betab$ to denote the exchangeable bootstrap estimator \cref{eq:defexboot_lfm} and use $\Pbt$ to denote the bootstrap empirical operator with exchangeable weights (\cref{as:exboot}).

    \textit{Step 1}. Show $\betab \toprob \betast$.
    The consistency of the bootstrap estimator is implied by uniform convergence of $H^b(\cd)$ to $H(\cdot)$ and uniqueness of $\betast$. For proof, we refer readers to the proof of Theorem 3.5 from \citet{gine1992bootstrap}.
    
    \textit{Step 2}. Show $\betab - \betast = O_p(1/\sqrt t)$.
    
    Define 
    \begin{align*}
        \Delta^\gam \defeq \cH\inv (\Pt - P) D_F(\cdot, \betast),
        \\
        \Delta^b \defeq \cH\inv(\Pbt - \Pt) D_F(\cd, \betast).
    \end{align*} 

    Let $ r_{1,F}(\cdot, \beta) = F(\cdot, \beta) - F(\cd,\betast) - D^*_F(\cdot) \tp (\beta - \betast)$, $D^*_F(\cd) = D_F(\cd,\betast) = \nabla F(\cd, \betast)$.
   
    We begin with the optimality of $\betab$ and then apply the definition of $r$.
    For ease of notation, we let $F(\beta) = F(\cdot,\beta)$.
    We have
    \begin{align}
       0 & \geq \Pexbt (F(\betab) - F(\betast)) 
       \notag
       \\
       & = (\Pexbt - \Pt + \Pt - P)(F(\betab) - F(\betast)) + P(F(\betab) - F(\betast))
        \notag
       \\
       & = 
       (\Delta^b + \Delta^\gamma) \tp \cH (\betab - \betast) 
       + P(F(\betab) - F(\betast))
       \\
        & \quad + (\Pexbt - \Pt + \Pt - P) r_{1,F}(\cdot, \betab) 
       \notag
       \\
       & \geq (\Delta^b + \Delta^\gamma + o_p(1)) \tp (\betab - \betast) 
       + c\cdot \|\betab - \betast \|_2\sq
       \label{eq:nuboot_rate}
    \end{align}
    where in the last inequality we used 
    (i)  $(\Pexbt - \Pt) r_{1,F}(\cdot, \betab) = o_p(\frac{1}{\sqrt t} + \|\betab-\betast\|_2) \cdot \|\betab-\betast\|_2 = o_p(1) \|\betab-\betast\|_2$ by \cref{lm:stocdiff_boot},
    (ii) $(\Pt - P) r_{1,F}(\cdot, \betab)    = o_p(1) \|\betab-\betast\|_2 $ by \cref{lm:stocdiff_boot},                 and
    (iii) $\beta \mapsto PF(\cdot, \beta)$ is locally strongly convex at $\betast$, and so there is a neighborhood of $\betast$ and a constant $c>0$ such that $P(F(\beta) - F(\betast)) \geq c \|\beta-\betast\|_2 \sq$ for all $\beta$ in this neighborhood.
    The expression \cref{eq:nuboot_rate} now becomes $0 \geq O_p(\invtroot) \|\betab -\betast\|_2 + c \|\betab - \betast\|_2\sq$. 
    Since the case $\betab - \betast = 0$ can be easily excluded, we divide both sides by $\|\betab - \betast\|_2$ and conclude that $(\betab - \betast) = O_p(1/\sqrt t)$.
    
    \textit{Step 3}. Find the asymptotic distribution.
    Since $\beta^b$ is the minimizer of $\Pexbt F$ over $\Rnp$, defining $\bar F(\beta) = F(\cdot, \beta) - F(\cd,\betast)$, we have

    \begin{align*}
        0 & \geq \Pexbt(\bar F (\beta^b) - \bar F(\Delta^b + \Delta^\gam + \betast))
        \\
        & = [(\Pexbt - \Pt) + (\Pt - P)] (\bar F (\beta^b) - \bar F(\Delta^b + \Delta^\gam + \betast)) + P(\bar F (\beta^b) - \bar F(\Delta^b + \Delta^\gam + \betast))
        \\
        & = (\Delta^b + \Delta^\gam) \tp \cH (\beta^b - \betast - (\Delta^b + \Delta^\gam  )) + \frac 12 \| \beta^b - \betast\|_\cH\sq - \frac 1 2 \|\Delta^b + \Delta^\gam \|_\cH\sq
         \\
        & \quad +  [(\Pexbt - \Pt) + (\Pt - P)] (r_{1,F}(\cd, \betab) - r_{1,F}(\cd, \Delta^b + \Delta^\gam + \betast))
        \\
        & = \frac12 \| \Delta^b + \Delta^\gam + (\beta^b - \betast)\|_\cH\sq + o_p(1/t)
    \end{align*}
    where in the last line we used 
    (i)  $(\Pexbt - \Pt ) r_{1,F}(\cd, \beta)   = o_p(1/t)$ for any random $\beta$ such that $\beta = \betast + O_p(1/\sqrt t)$ by \cref{lm:stocdiff_boot},
    (ii) $(\Pt - P ) r_{1,F}(\cd, \beta)      = o_p(1/t)$ for any random $\beta$ such that $\beta = \betast + O_p(1/\sqrt t)$ by \cref{lm:stocdiff_boot},
    and  (iii)         $P \bar F(\beta) = \frac12 \|\beta - \betast\|_\cH\sq + o(\|\beta - \betast\|_2\sq)$ for $\beta\to\betast$ due to $\nabla H(\betast) = 0$.
  
    Rearranging gives $t \| \Delta^b + \Delta^\gam + (\beta^b - \betast)\|_2\sq = o_p(1)$.
    Next       , using $ \betagam - \betast                                     = - \Delta^\gam + o_p(1/\sqrt t)$ \citep{liao2023stat}
    we have 
    \begin{align*}
        \sqrt t (\beta^b - \betagam) = -
        \sqrt t \Delta^b + o_p(1)
    \end{align*}
    By an exchangeable bootstrap CLT \citep{praestgaard1993exchangeably}, we know $\sqrt t \Delta^b \tod  c \cdot \cH\inv \cN(0, \E[\nabla F(\cd,\betast)\nabla F(\cd,\betast)\tp])$ conditional on almost all data sequence $\gamma$, where $c$ is the constant defined in \cref{as:exboot}. This concludes the proof.
    \end{proof}
  
\subsection{Proof of Numerical Bootstrap (\cref{thm:nuboot_fm} and \cref{thm:nuboot_fppepoor})}
\label{sec:proof:thm:nuboot_fm}

\nubootlfm*
\begin{proof}[Proof of \cref{thm:nuboot_fm} ] 
    We verify the assumptions in Theorem 4.1 of the numerical bootstrap paper \citep{li2020numerical}. 

    For \cref{thm:nuboot_fm}, let $B = \Rnpp$, $\betab$ be  $\betab_{\mathsf{nu,LFM}}$, and $\betast$ be the equilibrium pacing multiplier in Fisher market.
    We restate the assumptions in Theorem 4.1 of \citep{li2020numerical} in our notations.

    \begin{enumerate}[(i)]
        \item $H_t(\betagam) \leq \inf_B H_t + o_p (1/t)$, and $H^b_t(\betab) \leq \inf_B H^b_t  + o_p^*( \ept^2)$.
        \item $\betagam \toprob \betast$, and $\betab - \betagam = o_p^*(1)$.
        \item $\betast$ is an interior point of $B$.
        \item The class $\{ F(\cd, \beta) - F(\cd,\betast) : \|\beta-\betast\nmt \leq R \}$ is uniformly manageable.
        \item $H$ is twice differentiable at $\betast$ with positive definite Hessian $\cH$.
        \item The limit $ \Sigma(s,t)\defeq 
        \lim_{\epsilon\downarrow 0} \frac{1}{\epsilon \sq}
        \cov\big[
        F(\cdot,\betast + \epsilon s) -  F(\cdot,\betast), 
        F(\cdot,\betast + \epsilon t) - 
        F (\cdot,\betast)
        \big]$ exists for each $s,t$, 
        \item For all $\delta>0$ it holds $\lim_{\epsilon\downarrow 0} \frac1\epsilon \E[(F(\cdot, \betast + \epsilon s) - F(\cdot, \betast))\sq \indi (|F(\cdot, \betast + \epsilon s) - F(\cdot, \betast)|> \delta )] = 0$.
        \item Let 
        $G_R(\cdot) \defeq \sup_{ \|\beta-\betast\|_2 \leq R} | F(\cdot, \beta) - F(\cdot, \betast)| $.
        As $R\to0$,
        $ 
             \E[G_R\sq] = O(R\sq)
        $.
        \item $ \sqrt t \ept \to \infty $ and $\ept \downarrow 0$.
        \item $\E[G_R\sq \indi(R G_R > \eta )] = o(R\sq)
        $ for all $\eta > 0$.
        \item There is a neighborhood of $\betast$ such that $\E[|F(\cd, \beta) - F(\cd,\beta')|] = O(\| \beta - \beta'\|_2)$ for $\beta,\beta'$ in this neighborhood. 
        \footnote{In the \citet{li2020numerical} it is required $\E[|F(\cd, \beta) - F(\cd,\beta')|] = O(\| \beta - \beta'\|_2^2)$, which we think does not hold for the usual square root case.}
    \end{enumerate}

    Implicit in the paper, it is also required that the gradient of the population objective is zero at optimum. This is true for Fisher market.

    We now verify these conditions.
    Condition (i) holds because we consider exact minimizers. 
    Condition (ii): 
    The $\betagam - \betast \toprob 0$ part has been verified in \citet{liao2023stat}. It remains to show the $\beta^b- \betagam \toprob 0$ part. 
    Since $\beta \mapsto \Pbt F(\cd,\beta)$ converges to $H$ in probability pointwise and that $\beta \mapsto \Pbt F(\cd,\beta)$ is convex, it holds that the convergence is uniform over compact sets.  
    By uniqueness of $\betast$ it holds $\betab \toprob \betast$.
    Condition (iii) naturally holds for the linear Fisher market.
    Condition (iv) requires that there is an $R_0>0$ such that the function class $\{ F(\cdot,\beta) - F(\cdot ,\betast): \|\beta - \betast\|_2\leq R \}$ is uniformly manageable for all $R \leq R_0$.
    It holds because the class $\{ \theta \mapsto f(\theta, \beta) - f(\theta,\betast): \beta\in B\}$ is uniformly manageable, which is verified in 
    \citet{liao2023stat}.
    Condition (v) is assumed in 
    \nameref{as:twice_diff}.

    Consider the set $\prod_{i=1}^n [\betasti/2, 1]$. On this set, for all $\theta$, the function $\beta \mapsto F(\theta,\beta)$ is Lipschitz with parameter $L=\vbar + 2\sqrt n$ w.r.t.\ $\ell_2$ norm.
    Given this Lipschitz result, conditions (vi), (vii), (viii), and (x) follow easily.

    In condition (vi) 
    the function $\Sigma$ is $\Sigma(s,t) = s\tp \E[ \nabla \tilde F (\cdot, \betast) \nabla \tilde F(\cdot, \betast)\tp   ] t = s \tp \E[\nabla F(\cd,\betast) \nabla F (\cd, \betast) \tp ] t$ by using the dominated convergence theorem.

    Now we can invoke Theorem 4.1 from \citet{li2020numerical} and conclude $\epsilon_t\inv (\betab_{\mathsf{nu, LFM}} - \betagam) \topw \cJ_\LFM$.
\end{proof}

    \begin{proof}[Proof of \cref{thm:nuboot_fppepoor}] 
    For \cref{thm:nuboot_fppepoor}, let $B = [0,1]^n$.
    Let $\betab$ refer to $\betab_{\mathsf{nu,FPPE}}$, and $\betast$ be the equilibrium pacing multiplier in FPPE.
    As before, it is implicitly assumed in \citet{li2020numerical} that the 
    gradient of $H$ equals zero at $\betast$. This is true for FPPE if and only if all buyers spend their budgets.
    Theorem 4.2 from \citet{li2020numerical} requires that all conditions stated above except (i) and (iii) hold, and that $\betast$ uniquely minimizes $H$ over $B$, which is true in FPPE.
    Now we can invoke Theorem 4.2 from \citet{li2020numerical} and conclude $ \epsilon_t\inv (\beta^b_\mathsf{nu, FPPE} - \betagam) \topw \cJ_\FPPE $. 
    Note that under the assumption that all buyers spend their budgets, the limit distribution simplifies to 
        $\cJ_\FPPE = \argmin_{h \in C } G\tp h + \frac12 h\tp \cH h$ where $G\sim \cN(0, \E[\nabla F(\cd,\betast) \nabla F(\cd, \betast) \tp])$
    and $C         = \{h \in \Rn: h_i \leq 0, i\in I_{==}\} $ 
\end{proof}

\subsection{Proof of \cref{thm:hessian_est} and formal statement}
\label{sec:thm:hessian_est}
Below we start by giving a more formal version of \cref{thm:hessian_est} and then prove it.
\begin{theorem}[Hessian estimation] 
    Assume $H(\cd)$ is four times continuously differentiable in a neighborhood of $\betast$.
    Consider the finite difference estimate defined in \cref{eq:def_Gb_H} with differencing stepsize $\eta_t = o(1)$ and $\eta_t\sqrt t \to \infty$. 
    For some intermediate quantity $\check H _{k\ell}$, it holds 
      $    \hat \cH_{k\ell}  - \check \cH_{k\ell}=   o_p(\eta_t\sq + \tfrac{1}{\sqrt{t\eta_t\sq}}) + O_p(\frac{1}{\sqrt{t}})
      $, 
      $    \E[(\check \cH_{k\ell} - \cH_{k\ell})\sq] = 
       { \Theta(\eta_t ^ 4 + \frac{1}{t\eta_t\sq} )}+ o(\eta_t^4 + \frac{1}{t\eta_t\sq})
       $
      where the $O_p(1/\sqrt t)$ part does not depend on $\eta_t$. Proof in \cref{sec:thm:hessian_est}.
  \label{thm:hessian_est_formal}
  \end{theorem}
\begin{proof}[Proof of \cref{thm:hessian_est} and \cref{thm:hessian_est_formal}]
    The proof follows the idea in Lemma 2 from \citet{cattaneo2020bootstrap}. The main difference is their result is for cube-root asymptotics, while our setting is the usual square-root asymptotics.
    Define 
    \begin{align*}
        \check \cH_{k\ell} \defeq   (\tdifft H_t )(\betast)  , 
        \quad 
        \bar \cH_{k\ell} (\beta) \defeq  (\tdifft H ) (\beta)
    \end{align*}
    Then $\hat \cH_{k\ell} - \check \cH_{k\ell} = R + S$
    where 
    \begin{align*}
        R = \hat \cH_{k\ell} - \check \cH_{k\ell} - \bar \cH_{k\ell} (\betagam) + \bar \cH_{k\ell} (\betast) = \tdifft [(H_t - H)  (\betagam) - (H_t - H)(\betast)] ,
        \quad
        S = \bar \cH_{k\ell} (\betagam) - \bar \cH_{k\ell} (\betast) .
    \end{align*}
    It will be shown that 
    (i)  $R                                           = o_p(1/ (\sqrt t \eta_t)) $,
    (ii) $ S                                          = O_p(\frac1{\sqrt t}) + o_p(\eta_t\sq)$, and
    (iii) $\E[(\cH_{k\ell} - \check \cH_{k\ell})\sq ] = \Theta( \eta_t ^ 4 + \frac{1}{t\eta_t\sq}  ) + o(\eta_t^4 + \frac{1}{t\eta_t\sq})$.
  
    To show (i), by \cref{lm:stochdiff_of_F},
    \begin{align}
        \label{eq:stocheq_eps}
        \eta_t\inv \sqrt t \sup_{T_t} (P_t - P)( F(\cd, \betast + t_1 \eta_t) -F(\cd, \betast + t_2 \eta_t) ) = o_p(1)
    \end{align}
    where the supremum runs over the set $T_t=\{(t_1, t_2): \|t_1\|, \|t_2\|\leq M, \|t_1 - t_2\|\leq \delta_t\}$ for some $M>0$ and $\delta_t\downarrow 0$.
    Alternatively, the claim \cref{eq:stocheq_eps} can be proved following the proof of Theorem 4.1 from \citet{li2020numerical} or Lemma 4.6 from \citet{kim1990cube}.
    With this bound, letting $\delta = ( \betagam - \betast ) / \eta_t = o_p (1)$, terms such as
    \begin{align*}
       & (\Pt-P)(F(\cd, \betagam + \eta_t(e_k + e_\ell)) - F(\cd, \betast+ \eta_t(e_k + e_\ell)))
        \\
       & = (\Pt-P) (F(\cd, \betast+ \eta_t(\underbrace{\delta+e_k + e_\ell}_{t_1})) - F(\cd, \betast+ \eta_t(\underbrace{e_k + e_\ell}_{t_2})))
       \\
       & \leq \sup_{T_t} (P_t - P)( F(\cd, \betast + t_1 \eta_t) -F(\cd, \betast + t_2 \eta_t) )
       \\
       & = o_p(\eta_t / \sqrt t)
    \end{align*}
    can be upper bounded as above.
    And so $R = \frac1{\eta_t\sq} o_p(\eta_t/ \sqrt t) = o_p(1/\sqrt{t\eta_t\sq})$.
  
    To show (ii), by Taylor's theorem, $S =( \nabla_\beta \nabla_{k\ell} \sq H(\beta) | _{\beta=\betast})\tp (\betagam - \betast) + o_p(\eta_t\sq) = O_p(\frac1{\sqrt t}) + o_p(\eta_t)$, and the $O_p(1/\sqrt t)$ term does not involve $\eta_t$.
  
    Finally, to show (iii), we calculate the bias and variance of $\check \cH_{k \ell} $.
    Let $d(\cd) = 
    F(\cd, \betast + \eta_t ( e_k + e_\ell))
    - F(\cd, \betast + \eta_t ( - e_k + e_\ell))
    - F(\cd, \betast + \eta_t ( e_k - e_\ell))
    +  F(\cd, \betast + \eta_t ( - e_k - e_\ell))
    $.
    Then $\check \cH_{k\ell} = \frac{1}{4\eta_t\sq} P_t d(\cd)$.
    For the bias,
    following the proof of Lemma 2 from \citet{cattaneo2020bootstrap}, by Taylor's theorem,
    \begin{align*}
        \E[\check \cH_{k\ell} - \cH_{k\ell}] = \frac16 (\nabla\sq_{k} \nabla\sq_{k\ell} H(\betast)  +\nabla\sq_{\ell} \nabla\sq_{k\ell}  H (\betast)) \eta_t\sq + o (\eta_t\sq)
    \end{align*}
    To see this, let $g(\eps) = H(\betast+\eps h) - H(\betast)$. Then $g(0) = 0$. Also let $\cH(\b) = \nabla \sq H(\b)$.
    Now $g^{(1)} (\eps)= \nabla H(\betast + \eps h)\tp h$, $g^{(2)} (\eps)=  h \tp \cH (\betast + \eps h) h$,
    $g^{(3)} (\eps) = \sumiton h_i\sq h\tp \nabla_\b \cH_{ii} (\betast + \eps h) + 2 \sum_{i<j} h_{i} h_j h\tp \nabla_\b \cH(\betast + \eps h) $, and $g^{(4)} =  \sumiton h_i\sq h\tp \nabla_\b\sq \cH_{ii} (\betast + \eps h) h + 2 \sum_{i<j} h_{i} h_j h \tp \nabla_\b\sq  \cH(\betast + \eps h)  h $.
  
    For the variance,  note $\var(\check \cH_{k\ell}) = \frac{1}{16 t \eta_t^4} \var( d(\cd)) = \frac{1}{16 t\eta_t^4} \E[ d(\cd)\sq] + O(1/t)  $.
    Next, 
    \begin{align*}
        \eta_t ^{-2} \E[ ( F(\cd, \betast + \eta_t ( e_k + e_\ell)) - F(\cd, \betast))\sq]
        \to ( e_k + e_\ell)\tp \E[ \nabla F(\cd, \betast)\ot ]( e_k + e_\ell)  
    \end{align*}
    and so $\E[d(\cd)\sq] = \Theta(\eta_t\sq)$.
    Conclude that 
    \begin{align*}
        \var(\check \cH_{k\ell})
         =  \Theta( \frac{1}{t\eta_t\sq} ) + O(1/t) .
    \end{align*}
  \end{proof}
  
\subsection{Proof of Proximal Bootstrap (\cref{thm:prboot_fm} and \cref{thm:prboot_fppepoor})}
\label{sec:proof:thm:prboot_fm}

\begin{proof}[Proof of \cref{thm:prboot_fm}]
As $t \to \infty$, $\ept\inv (\betab - \betagam)= - \hat \cH \inv G^b$ with probability approaching 1 due to $\ept = o(1)$ and $\ept \sqrt t \to \infty$.
Next, $G^b = \sqrt t (\Pbt - \Pt) D_F(\cd,\betagam) = \sqrt t (\Pbt - \Pt)( D_F(\cd,\betagam) - D_F (\cd,\betast) ) + \sqrt t (\Pbt - \Pt) D_F(\cd,\betast) = \sqrt t (\Pbt - \Pt) D_F(\cd,\betast)  + o_p(1) \tod \cN(0, \E[\nabla F(\cd, \betast) \nabla F(\cd, \betast)\tp])$ 
conditional on data,
by \citet[Lemma 4.1]{wellner1996bootstrapping}.
Also $\hat \cH \toprob \cH$ and by continuity of matrix inverse, $\hat \cH \inv \toprob \cH \inv$.
We conclude $\eps_t \inv (\betab - \betagam)$ converges in distribution to $\cJ_{\LFM}$ conditionally in probability.
\end{proof}

\begin{proof}[Proof of \cref{thm:prboot_fppepoor}]
   We present proof following the idea in the working paper of \citet{li2023proximal}.
    In fact, most of the conditions required in that paper have been established in \citet{liao2023stat}, such as stochastic equicontinuity of certain processes.

    Let $B = [0,1]^n$. Let $\betab$ refer to $\betab_{\mathsf{nu,FPPE}}$, and $\betast$ be the equilibrium pacing multiplier in FPPE.
    It is assumed in \citet{li2023proximal} that the 
    gradient of $H$ equals zero at $\betast$. This is true for FPPE if and only if all buyers spend their budgets, i.e., $I_{=>} = \emptyset$.

Step 1. Show $\betab \toprob \betast$. 
    Since $\dtt \to 0$ and $G^b = O_p(1)$, we have $\dtt G^b = o_p(1)$. Then for each $\beta \in \Rn$
    \begin{align*}
        & \dtt (G^b)\tp \b + \frac12 \| \betagam - \beta \|_{\hat \cH} \sq + \chi(\beta\in [0,1]^n) 
        \\ 
        & =  \frac12 \| \betast - \beta \|_{\hat \cH} \sq + \chi(\beta\in[0,1]^n)
         + \bigg(\dtt (G^b )\tp \b + (\betast - \beta) \tp \hat \cH (\betagam - \betast)  + \frac12 \| \betagam - \betast\|_{\hat \cH}\sq \bigg)
        \\
        & =  \frac12 \| \betast - \beta \|_{\hat \cH} \sq + \chi(\beta\in [0,1]^n)
        + o_p(1)
        \\
        & \toprob 
        \frac12 \| \betast - \beta\|_{\cH} \sq + \chi(\beta \in [0,1]^n ) 
    \end{align*}
    By convexity, it implies uniform convergence on compact sets in probability.
    Since $\betast$ uniquely minimize $\b \mapsto \| \betast - \beta\|_{\cH} \sq + \chi(\beta \in [0,1]^n)$, we conclude $\betab \toprob \betast$ \citep[Theorem 2.7]{newey1994large}.

Step 2. Identify the limit distribution of $\betab$.
Note that when $I_{=>} = \emptyset$, $\E[\nabla F(\cd,\betast)] = 0$.
    Define $X_t(h) =(G^b)\tp (h + \frac{\betast - \betagam}{\dtt}) + \frac12 \|h + \frac{\betast - \betagam}{\dtt}\|_{\hat \cH} \sq  $.
    We first show $X_t(h) \rightsquigarrow G\tp h + \frac12 h\tp \cH h$ in $\ell^\infty (K)$ for any compact $K\subset \Rn$, where $G\sim \cN(0 ,\E[\nabla F(\cd, \betast)\nabla F(\cd, \betast) \tp])$. The proof is identical to the proof of \cref{claim:bsprocess} and is omitted here.
    Next, by a change of variable $h = \eps_t \inv (\beta - \betast)$, the inclusion $\beta \in [0,1]$ becomes $h \in({[0,1]^n - \betast})/{\eps_t} $, and
    \begin{align*}
        \eps_t \inv (\betab - \betast) = \argmin_{h \in ({[0,1]^n - \betast})/{\eps_t}} X_t(h)
        \tod \argmin_{h \in \Rn: h_i \leq  0, i \in I_{==}}  G\tp h + \frac12 h\tp \cH h = \cJ_\FPPE
    \end{align*}
    where the last equality uses $I_{=>} = \emptyset$ and thus $ ({[0,1]^n - \betast})/{\eps_t}\toepi \{h \in \Rn: h_i \leq 0, i \in I_{==} \}$.
    We conclude the proof of \cref{thm:prboot_fppepoor}.
\end{proof}

\subsection{Proof of \cref{thm:failbs}}
\label{sec:proof:thm:failbs}
The limit FPPE is $\betast_1 = 1$ and $\deltast_1 = 0$. 
The observed FPPE is $\betagam_1 = \min\{1, 1/{\vbart}\}$ where $\vbart = \frac1t \sumtau v_1(\thetau)$. 
The bootstrapped FPPE  \cref{eq:bsexample_eg}
 is $\beta^b_1 = \min \{ 1 , 1/ \bar{v} ^{t,\boot}\}$ where $\bar{v} ^{t,\boot} = \frac1t \sumtau v_1(\theta^{\tau,\boot})$.

 First, we derive the limit distribution of the observed FPPE. We have
 \begin{align*}
     \sqrt{ t} (\betagam_1 - 1) = \frac{1}{\vbart} \min \{ \sqrt t (1 - \vbart), 0 \} \tod \min \{Z, 0\} = \cJ_\FPPE.
 \end{align*}
 where $Z \sim \cN (0, \var(v_1))$. In the above we used $\vbart \toprob \E[v_1]=1$, Slutsky's Theorem, $\sqrt{t}(1-\vbart) \tod Z$, and the continuous mapping theorem.

 Now we analyze the limit distribution of  the bootstrapped FPPE.
 Define the set $A_c = \{  \limsup_t \sqrt{t} (1 - \vbart ) > c\}$ for any $c > 0$. By the law of the iterated logarithm, 
 \begin{align*}
    \P\bigg( 
        \limsup_{t\to\infty}\frac{\sqrt{t} (1-\vbart)}{\sqrt{2\log\log t}} = 1
    \bigg) = 1,
 \end{align*}
and thus $\P(A_c) = 1$ for all $0<c<\infty$.
Note that it holds $ \vbartboot - 1\toprob 0$ and $\sqrt{t} (\vbartboot - \vbart) \tod \cN(0,\var(v_1))$ conditional on observed items (by triangular-array versions of the law of large numbers and the central limit theorem, see Theorem 2.2.6 and Theorem 3.4.10 from \citet{durrett2019probability}). 
 On the event $A_c\,$, we can choose a subsequence $\{t_k\}_k\,$, such that $ \sqrt{t_k} (1 - \bar v ^{t_k} )  \geq c$ for all $k$. 
 Now let $t$ be an element of this subsequence. Then we have
 \begin{align*}
    \sqrt{t} (\beta^b_1 - \betagam_1) & = \sqrt{t} (\min\{1, 1/\vbartboot\} - \min\{1,1/\vbart\} )
    \notag
    \\
    &\geq 
    \sqrt{t} (\min\{1, 1/\vbartboot\} - 1 )
    \notag
    \\
    &= \frac{1}{\bar v ^{t, \boot}} \min \{ 0, \sqrt{t} ( \vbart - \vbartboot + 1 - \vbart) \}
    \notag
    \\ 
    &\geq 
    \frac{1}{\bar v ^{t, \boot}} \min \{ 0, \sqrt{t} ( \vbart - \vbartboot ) + c \}
    \\ 
    &\tod 
    \min \{ 0, Z + c \} \geq \min \{0, Z\},
    \notag
 \end{align*} 
 where we used Slutsky's theorem for the convergence in probability.
 The last inequality is strict with strictly positive probability. 
 We conclude that
 the standard multinomial bootstrap $\sqrt{t}(\beta^b - \betagam)$ fails to converge to the desired distribution~$\cJ_\FPPE$.

\subsection{Proof of \cref{thm:adaptive_boot}}
\label{sec:proof:thm:adaptive_boot}
\begin{proof}[Proof of \cref{thm:adaptive_boot}]
    Define the estimated critical cone 
    \begin{align*}
        \hat C \defeq \{ h : h_i = \frac{1-\betasti}{\dwt} \text{ for } i \in \hat I _{=>}  \}
    \end{align*}
    Recall under \nameref{as:constraint_qualification} the critical cone is $C = \{h \in \Rn :  h_i = 0 , i \in I_{=>} \}$.

    \textit{Step 1}. We show that the critical cone is correctly estimated in the sense that 
    \begin{align}
        \label{eq:coneepiprob}
        \chi( \cdot \in \hat C) \toepi \chi ( \cdot \in C)
        \text{ \quad in probability}
    \end{align}    
    The claim is equivalent to $d_{AW} (\hat C, C) \toprob 0$, where $d_{AW}$ is defined in \cref{eq:defdaw}.
    For any $\epsilon > 0$, the event $\{ d_{AW} (\hat C, C) > \epsilon\}$ is equivalent to $ \{ \hat I_{=>} \neq I_{=>}  \}$.
    First we bound $\P( \hat I_{=>} \neq I_{=>})$.    
    \begin{align*}
        \P( \hat I_{=>} \neq I_{=>}) 
        & = \P(\exists i \in I_{=>}, 1-\betagami > \dwt, \text{ or } \exists i \in \Ic,  1-\betagami < \dwt)
        \\
        & \leq \sum_{i\in I_{=>}} \P( 1-\betagami > \dwt)
        + \sum _{ i \in \Ic} \P( 1-\betagami < \dwt)
    \end{align*}
    For $i\in I_{=>}$, by $\betagam_i - 1 = o_p(\frac{1}{\sqrt t})$, we have $\P( 1-\betagami > \dwt) = \P( o_p(1) > \dwt \sqrt t) \to 0$ since $\dwt \sqrt t \to c > 0$. 
    For $i \in \Ic$, since $\betasti - 1 < 0$, $\P( 1 - \betagami < \dwt) = \P(\betasti - \betagami < \betasti - 1 + \dwt) = \P(o_p(1) <\betasti - 1 + \dwt ) \to 0$ by $\dwt \downarrow 0$.
    We conclude $\P( \hat I_{=>} \neq I_{=>}) \to 0$
    and so
    $ \chi( \cdot \in \hat C) \toepi \chi ( \cdot \in C)
    $ in probability.

    \textit{Step 2.}
    Next, we show $\betab \toprob \betast$.
    Since $\dtt \to 0$ and $G^b = O_p(1)$, we have $\dtt G^b = o_p(1)$. Then for each $\beta \in \Rn$
    \begin{align*}
        & \dtt G^b + \frac12 \| \betagam - \beta \|_{\hat \cH} \sq + \chi(\beta\in \hat B) 
        \\ 
        & =  \frac12 \| \betast - \beta \|_{\hat \cH} \sq + \chi(\beta\in \hat B)
         + \bigg(\dtt G^b + (\betast - \beta) \tp \hat \cH (\betagam - \betast)  + \frac12 \| \betagam - \betast\|_{\hat \cH}\sq \bigg)
        \\
        & =  \frac12 \| \betast - \beta \|_{\hat \cH} \sq + \chi(\beta\in \hat B)
        + o_p(1)
        \\
        & \toprob 
        \frac12 \| \betast - \beta\|_{\cH} \sq + \chi(\beta \in B ' ) 
    \end{align*}
    where $B ' = \{\beta \in [0,1]^n: \betai = 1, i \in I_{=>}  \}$.
    By convexity, it implies uniform convergence on compact sets in probability.
    Since $\betast$ uniquely minimize $\| \betast - \beta\|_{\cH} \sq + \chi(\beta \in B)$, we conclude $\betab \toprob \betast$ \citep[Theorem 2.7]{newey1994large}.

    \textit{Step 3}. Identify the limit distribution of $\betab$.

    Note $\dtt\inv (\betab - \betagam) = \dtt\inv (\betab - \betast) + \dtt\inv (\betast - \betagam) = \dtt \inv (\betab - \betast) + o_p(1)$, it suffices to show $\dtt \inv (\betab - \betast) \topw \cJ_\FPPE$.

    First with a change of variable $h = \dtt\inv (\beta - \betast)$ and so $\dtt\inv (\beta-\betagam ) = h +\dtt\inv (\betast - \betagam)$, dividing the objective function by $\dtt\sq$, the bootstrap estimator in \cref{eq:adboot_fppe}, with probability approaching one, can be written as
    \begin{align*}
        \frac{\betab - \betast}{\dtt} = \argmin_{h \in \Rn} \bigg\{ 
            \underbrace{(G^b)\tp (h + \frac{\betast - \betagam}{\dtt}) + \frac12 \|h + \frac{\betast - \betagam}{\dtt}\|_{\hat \cH} \sq }_{:= X_t(h)} 
        + 
        \chi(h \in \hat C) \bigg\}
    \end{align*} 

    \begin{claim}  
        \label{claim:bsprocess}
        $X_t(h) \rightsquigarrow (\topw)G\tp h + \frac 12 h \tp \cH h$ in $\ell^\infty (K)$ for any compact $K \subset \Rn$.
    \end{claim}

    Combining \cref{claim:bsprocess} and \cref{eq:coneepiprob},
    based on the argument in the proof of \citet[Theorem 4.2]{li2020numerical}, we have $\dtt\inv (\betab - \betast) \topw \argmin_{h\in C} G\tp h + \frac12 h\tp \cH h $, which is exactly $\cJ_\FPPE$.
\end{proof}

\begin{proof}[Proof of \cref{claim:bsprocess}]
    First, we show for any compact set $K \in \Rn$, 
    \begin{align}
        \label{eq:smaller_terms}
        \sup_{h \in K} |X_t(h) - ( (G^b) \tp h + \frac12 \|h\|_{ \cH} \sq  
       )| = o_p(1)
    \end{align}
    Note the LHS can be upper bounded by 
    \begin{align*}
        (G^b)\tp (\betagam - \betast)\dtt\inv 
        + \frac12 \|\betagam-\betast\|_{\cH}\sq \dtt^{-2}
        +  \sup_K h\tp \hat \cH (\betagam - \betast)\dtt\inv
        + \sup_K h\tp (\hat \cH - \cH) h
    \end{align*}
    We          just need to show $(G^b)\tp (\betagam-\betast) = o_p(\dtt)$,
    $\|\betagam - \betast\|_{\hat \cH}                         = o_p(\dtt\sq)$ and
    $\sup_K     h\tp \hat \cH (\betagam - \betast)             = o_p(\dtt)$ and
    $\sup_K     |h\tp \cH h - h \tp \hat \cH h|                = o_p(1)$ (\cref{thm:hessian_est}).
    This        holds by $G^b                                  = O_p(1)$,                $\betagam - \betast = O_p(1/\sqrt t)$,
    $\|         \cH - \hat \cH\|_2                             = o_p(1)$ and $1/\sqrt t = o(\dtt)$.

    Next we show for any compact $K \subset \Rn$, 
    \begin{align}
        \label{eq:compactconv}
       (G^b) \tp h + \frac12 \|h\|_{\hat \cH} \sq  
       \rightsquigarrow (\topw)
       G\tp h + \frac12 \|h\|_\cH\sq
       \text{\quad in $\ell^\infty(K)$}
    \end{align}
    where $G \sim \cN(0, \cov(\nabla F(\cd,\betast)))$.
    For unconditional convergence it suffices to show $G^b \tod G$ and $\hat \cH \toprob \cH$. In \cref{thm:hessian_est} is has been shown that $\hat \cH - \cH = o_p(1)$.
    Note $G^b = \sqrt t (\Pbt - \Pt) D_F(\cd,\betagam) = \sqrt t (\Pbt - \Pt)( D_F(\cd,\betagam) - D_F (\cd,\betast) ) + \sqrt t (\Pbt - \Pt) D_F(\cd,\betast) = \sqrt t (\Pbt - \Pt) D_F(\cd,\betast)  + o_p(1)$ by \citet[Lemma 4.1]{wellner1996bootstrapping}. We conclude $G^b \tod G$ and $X_t \rightsquigarrow h\tp G + \frac12 h\tp \cH h$.
    To show conditional convergence, we use a conditional central limit theorem and obtain $G^b \tod G$ conditional on data.
\end{proof}

\subsection{Proof of \cref{thm:generalFPPE_coverage}}
\label{sec:proof:thm:generalFPPE_coverage}

\begin{proof}[Proof of \cref{thm:generalFPPE_coverage}]
Let $T^{b,\infty}$ be the conditional limit distribution of $T^b$, and $T^ \infty$ be the limit distribution of $T^\gam (\betast, \deltast)$. 
    The proof relies on the following result.
    For two real-valued random variables $X$ and $Y$, we say $X$ is stochastically dominated by $Y$, denoted $X \leq_{st} Y$ if $\P(X>x) \leq \P(Y > x)$ for all $x\in \R$.

    \begin{theorem}
        \label{thm:generalFPPE}
        $T^b  \topw - \inf_{h \in \Rn} 
       ( G \tp h + \frac12 h\tp \cH h) =: T^{b,\infty}$, and 
        $ T^\gam(\betast, \deltast) 
        \tod -\inf_{h \in \B_\kappa } 
        ( G \tp h + \frac12 h\tp \cH h) =: T$.
        For all $\kappa \in (0, \infty)$, $T^\infty \leq_{st} T^{b,\infty}$.
        When $\kappa = \infty$, $T^\infty = T^{b,\infty}$.
    \end{theorem}
    Asymptotic coverage follows from \citet{beran1984}. This completes the proof.
\end{proof}

\begin{proof}[Proof of \cref{thm:generalFPPE}]
    We study the asymptotic distributions of $T^\gam(\betast, \deltast)$ and $T^b$.
    Recall $\nabla H (\betast) = - \deltast$.

    Step 1. We will show 
    \begin{align*}
         t (L_t(\betast +\tfrac{h}{\sqrt t}, \deltast) - L_t (\betast, \deltast)) 
        \rightsquigarrow h \tp G + \frac12 h \tp \cH h
    \end{align*}
    in $\ell^\infty ( K)$ for any compact $K \subset \Rn$.
    \begin{align*}
        & t (L_t(\betast + \tfrac{h}{\sqrt t}, \deltast) - L_t (\betast, \deltast)) 
        \\
        & = t (H_t(\betast + \tfrac{h}{\sqrt t}) - H_t(\betast) + (\deltast) \tp (\tfrac{h}{\sqrt t}) )
        \\
        & =
          \sqrt t (P_t - P) (\nabla F(\cd, \betast) )  \tp h+ \frac12 h\tp \cH h + o_p (1)
        \\
        & \rightsquigarrow h \tp G + \frac 12 h \tp \cH h \text{ in $\ell^\infty(K)$}
    \end{align*}
    where $G\sim \cN(0, \cov(\nabla F(\cd, \betast)))$, $o_p(1)$ term is uniform over $h \in K$ by \cref{claim:htquadraticexpansion}.
    Applying a continuous mapping theorem
    \begin{align*}
        T^\gam(\betast, \deltast) 
        & = -\inf_{h \in \B_{\kappa} } t (L_t(\betast + \tfrac{h}{\sqrt t}) - L_t (\betast))
        \\
        & \tod -\inf_{h \in \B_\kappa } 
        ( G \tp h + \frac12 h\tp \cH h) =: T
    \end{align*}
   
    Step 2. 
    Recall 
    \begin{align*}
        X^b (\beta) & =   (\dtt (G^b)\tp (\beta - \betagam) + \tfrac12 \|\betagam- \beta\|_{\hat \cH} \sq)  / (\dtt)\sq \;,
        \\ 
        T^b 
        & 
        =  - \inf_{\beta \in \Rnp}(\dtt (G^b)\tp (\beta - \betagam) + \tfrac12 \|\betagam- \beta\|_{\hat \cH} \sq)  / (\dtt)\sq\;.
    \end{align*}
    In \cref{claim:bsprocess}, we have shown 
    \begin{align*}
        X^b(\betast +  \dtt h) \topw G \tp h + \frac12 h\tp \cH h
    \end{align*}
    in $\ell^\infty ( K)$ for any compact $K \in \Rn$. 
    Next we study the asymptotic distribution of $T^b$.
    \begin{align*}
       T^b 
       & = -\inf_{\beta \in \Rnp} X^b(\b) 
       \\
       & = -\inf_{h \in (\Rnp - \betast)/{\dtt}} X^b(\betast +   \dtt h)
       \\
       & \topw - \inf_{h \in \Rn} 
       ( G \tp h + \frac12 h\tp \cH h) =: T^{b,\infty}
    \end{align*}
    where the last line follows due to $\betast$ lying in the interior of $\Rnp$, and a bootstrap version of extended continuous mapping theorem; see Proposition 10.7 in \citet{kosorok2008introduction}.
    We can see for each draw of $G$, we have the dominance relationship $T^{b,\infty} \geq T$.
    We conclude the $(1-\alpha)$-quantile of $T^{b,\infty}$ is greater than or equal to that of $T$.

    The claim that when $\kappa = \infty$, $T^{b} \topw T^\infty$ is obvious.
    \begin{remark}
        One could also use the statistic 
        \begin{align*}
            T^\gamma (\b) = \inf_{ 0 \leq \delta \leq b, \delta \tp (1_n - \b) = 0} \bigg(L_t (\b,\delta) - \inf _ {h \in \B_\kappa} L_t (\b + h/\sqrt t, \delta)\bigg)
        \end{align*}
    and the region $\{ \b \in (0, 1]^n: T^\gamma (\b) \leq \iota \}$ to do inference on just $\b$.
    Noting $ T^\gamma(\betast) \leq L_t (\betast,\deltast) - \inf _ {h \in \B_\kappa} L_t (\betast + h/\sqrt t, \deltast)$, we can estimate an upper bound of the quantile of its limit distribution by bootstrap.
\end{remark}
\end{proof}